\documentclass{amsart}

\newtheorem{theorem}{Theorem}[section]
\newtheorem{lemma}[theorem]{Lemma}
\newtheorem{remark}[theorem]{Remark}
\newtheorem{proposition}[theorem]{Proposition}

\newtheorem{example}[theorem]{Example}

\usepackage{amsmath, amsthm, amssymb}
\usepackage{dsfont}
\usepackage[
  backend = biber,
  style = numeric,
  sorting = nyt,
  giveninits = true,
  maxbibnames = 99,
  maxcitenames = 2
]{biblatex}
\usepackage{hyperref}
\usepackage{enumerate}
\usepackage{tikz}
\usepackage[skip = 0.5\baselineskip]{caption}

\DefineBibliographyStrings{english}{
  page = {},
  pages = {},
  in = {}
}
\DeclareFieldFormat[article]{volume}{\mkbibbold{#1}}
\DeclareFieldFormat[article]{number}{\mkbibparens{#1}}

\renewbibmacro*{volume+number+eid}{%
  \printfield{volume}%
  \setunit{\addspace}%
  \printfield{number}%
  \setunit{\addcomma\space}%
  \printfield{eid}}

\newcommand{\cemph}[1]{\emph{\color{red}#1}}

\newcommand\R{\mathbb{R}}
\newcommand\B{\mathbb{B}}

\newcommand\CK{\mathcal{K}}
\newcommand\CE{\mathcal{E}}

\newcommand\bone{\mathds{1}}
\newcommand{\dd}{\,\mathrm{d}}

\DeclareMathOperator{\conv}{conv}
\DeclareMathOperator{\pos}{pos}
\DeclareMathOperator{\aff}{aff}
\DeclareMathOperator{\lin}{lin}
\DeclareMathOperator{\inte}{int}
\DeclareMathOperator{\bd}{bd}
\DeclareMathOperator{\vol}{vol}
\DeclareMathOperator{\argmax}{argmax}

\newcommand{\diff}[3]{D_{#1}^{#2} #3}
\newcommand{\diag}[2]{\Delta^{#1} #2}
\newcommand{\lyz}[4]{C_{#3} (#4)}
\newcommand{\interval}[4]{I_{#3} (#4)}
\newcommand{\sumfunc}[5]{\sigma_{#3}(#4,#5)}

\newcommand{\simplexscale}[2]{#1(#2)}
\newcommand{\simplexscaling}[2]{\mu_{#1}(#2)}

\newcommand{\simplex}[1]{\mathcal{S}^{#1}}

\newcommand{\embedding}[2]{#1 \times \{#2\}}
\newcommand{\coneembedding}[3]{\conv((\embedding{#1}{#2}) \cup \{#3\})}

\newcommand{\cartesian}[2]{#1^{\times #2}}

\title[Tight stability estimates near the simplex]{Tight Stability Estimates Near the Simplex and Applications for the Banach--Mazur Distance and Rogers--Shephard-Type Inequalities}

\author[R. Brandenberg]{Ren\'e Brandenberg}
\author[B. Gonz\'alez Merino]{Bernardo Gonz\'alez Merino}
\author[F. Grundbacher]{Florian Grundbacher}
\date{September 14, 2026}

\begin{document}

\begin{abstract}
We establish a tight stability estimate for the Minkowski asymmetry near its maximal value, improving earlier results in both the range for the admissible error and the strength of the estimate. More precisely, if an $n$-dimensional convex body $K$ has Minkowski asymmetry $s(K) \geq n-\varepsilon$ for $\varepsilon \in [0,1)$, then its Banach--Mazur distance to the $n$-simplex is at most
\[ 1 + \varepsilon + \frac{\varepsilon^2}{2(1-\varepsilon)}. \]
This dimension-independent estimate is sharp to the linear order in $\varepsilon$, including the constant.

We apply this estimate to several problems. First, we prove a stability result for the maximal Banach--Mazur distance to the Euclidean ball, improving a previous estimate to the optimal linear order. As a key ingredient, we verify the conjecture that every convex body $K$ contains a translated copy of its volume-minimal circumscribed ellipsoid scaled down by a factor $\sqrt{n s(K)}$.
Second, we prove a sharp common generalization of Schneider's higher-order Rogers--Shephard inequality and the $L_p$-Rogers--Shephard inequality, and establish a stability result of the optimal linear order. These results are based on the recent positive answer to the inequality part of the higher-order Godbersen conjecture and the accompanying proof of the $L_p$-Rogers--Shephard inequality. Finally, we improve upper bounds for the diameter of the Banach--Mazur compactum in fixed dimensions.
\end{abstract}

\maketitle

\section{Introduction}

Stability results are the natural next step in understanding the behavior of some geometric inequality at its extremal value once a characterization of its equality case is known.
Many important stability results have been established over the years,
see, e.g., \cite{2011:BallBöröczky,2005:Böröczky,2010:Böröczky,2026:BöröczkyPatsalosSaroglou,2025:FigalliHintumTiba,2010:FigalliMaggiPratelli,2008:FuscoMaggiPratelli}.
We draw special attention to some of the results where simplices uniquely reach the extremal value of an inequality:
B\"or\"oczky \cite{2005:Böröczky} obtained the stability of the Rogers--Shephard inequality and the Zhang projection inequality,
Kobos \cite{2019:Kobos} provided a stability estimate for the upper bound on the Banach--Mazur distance to smooth symmetric convex bodies, 
Schneider \cite{2009:Schneider} proved stability results for the Bohnenblust and Leichtweiss inequalities,
and Lucardesi and Zucco \cite{2025:LucardesiZucco} recently showed stability versions of P\'al's inequality.

A crucial ingredient common to several of the explicitly mentioned results is the stability of the Minkowski asymmetry near the simplex.
Denoting by $\CK^n := \{ K \subset \R^n : K$ compact, convex, $ \inte(K) \neq \emptyset \}$ the family of \cemph{(convex) bodies},
let us recall that the \cemph{Minkowski asymmetry} $s(K)$ of $K \in \CK^n$
is the smallest positive factor by which $K$ can be rescaled while still containing a translation of $-K$.
It is well-known that $s(K) \leq n$, with equality if and only if $K$ is an $n$-dimensional simplex (see \cite{1963:Grunbaum}),
which we denote as $\simplex{n}$ if we do not want to specify any particular simplex.
Several stability estimates have been established for the simplex as the unique maximizer in this context, see \cite{2005:Böröczky,2005:Guo,2009:Schneider}.
For almost two decades, the strongest of them was given by Schneider
in the following form:
If $K \in \CK^n$ is a convex body with
$s(K) \geq n - \varepsilon$ for some $\varepsilon \in [0,\frac{1}{n})$,
then
\begin{equation}\label{eq:Schneider}
	d_{BM}(K,\simplex{n})
	\leq 1 + \frac{n+1}{1-n\varepsilon} \cdot \varepsilon = 1 + (n+1) \varepsilon + \sum_{k = 2}^\infty (n+1) n^{k-1} \varepsilon^k.
\end{equation}
Above,
$d_{BM}(K,L)$ denotes the Banach--Mazur distance between two bodies $K, L \in \CK^n$ (see below).

Our first main result is an improvement of \eqref{eq:Schneider} that entirely removes the dependence on the dimension $n$,
meaning the allowable error $\varepsilon$ is increased from $\frac{1}{n}$ to constant size $1$ and the upper bound becomes free of $n$ as well.
We additionally show that the improved upper bound is optimal up to the linear order.

\begin{theorem}\label{thm:main}
Let $K \in \CK^n$ with $s(K) \geq n-\varepsilon$ for some $\varepsilon \in (0,1)$.
Then,
\[
    d_{BM}(K,\simplex{n})
    \leq 1 + \varepsilon + \frac{\varepsilon^2}{2(1-\varepsilon)}
    = 1 + \varepsilon + \sum_{k = 2}^\infty \frac{\varepsilon^k}{2}.
\]
Further, for all $n \geq 2$ and $\varepsilon \in [0,1]$,
there exists $K \in \CK^n$ with $s(K) = n - \varepsilon$ and $d_{BM}(K,\simplex{n}) = 1 + \varepsilon$.
\end{theorem}

The basis for the improvements compared to \eqref{eq:Schneider} is two-fold:
When a convex body $K$ is scaled by $s(K)$, then a translated copy of $-K$ inside must share a number of its boundary points.
Starting from structural observations due to Klee \cite{1953:Klee} and Böröczky \cite{2005:Böröczky}, we show how a simplex formed from such contact points can be used to obtain a sufficiently small superset of $K$ constructed from certain $(n+1)^2$ points, 
namely the $n+1$ vertices of the simplex and $n$ additional points near each vertex.
Compared to previous, comparable approaches based on supersets of $K$ \cite{2005:Böröczky},
we have more global control over our constructed superset,
which is the main reason behind the increase of the range for $\varepsilon$ from $[0,\frac{1}{n})$ to $[0,1)$.
The second major difference in our approach compared to previously published methods is that we take advantage of arbitrary copies of the approximating simplices instead of insisting on specific translations.
While resulting in more technical challenges, this change is crucial for removing the dimension dependence also from the estimate.

We provide Example~\ref{ex:asymmetry_simplex_tightness} showing that the simplices we construct can in general not lead to a better upper bound on the Banach--Mazur distance to the simplex than the estimate in Theorem~\ref{thm:main}.
Accordingly, further improvements require other methods.

Let us also point out that a slightly weaker estimate than the one in
Theorem~\ref{thm:main} was obtained very recently by Bakaev and Yehudayoff in the preprint \cite{2026:BakaevYehudayoff}.
It provides the same range of $\varepsilon \in [0,1)$,
though the upper bound is worse by a factor $2$ in the constant for all orders in $\varepsilon$ starting from the quadratic order.
The approach in \cite{2026:BakaevYehudayoff} is fundamentally different from ours, focusing on simplices circumscribed about $K$ (see Section~\ref{sec:StabilityNearSimplex} for some details).

Theorem~\ref{thm:main} should be viewed as a general stability principle.
Whenever the Minkowski asymmetry can be introduced to refine a sharp inequality with simplices as the unique extremal bodies,
Theorem~\ref{thm:main} immediately converts this into a stability result near the simplex.
Alternatively, it can help improve absolute bounds in situations where the space of convex bodies is parametrized by the Minkowski asymmetry.
Examples where the previously known, weaker versions of the theorem are applied in such situations can be found in \cite{2005:Böröczky,2023:BrandenbergDichterGonzalezMerino,2026:CaDiGoGrJaRu,2021:LianWu,2009:Schneider}.
The often rather straightforward ways in which substantial improvements can be obtained indicate that the stability for the Minkowski asymmetry near the simplex has so far been underutilized in this context,
especially now that Theorem~\ref{thm:main} removes the dependence on the dimension from the estimate.

We provide three
further examples of open problems where Theorem~\ref{thm:main} can be involved advantageously.
The first one concerns the Banach--Mazur distance to the Euclidean ball.
A well-known consequence of the John ellipsoid theorem \cite{1948:John} is that every convex body $K \in \CK^n$ satisfies $d_{BM}(K,\B_2^n) \leq n$,
where $\B_2^n$ is the \cemph{Euclidean unit ball} in $\R^n$,
and that this bound may be reduced to $\sqrt{n}$ when $K$ is symmetric.
Given the fundamental role that these estimates play
for the theory of the Banach--Mazur distance and the local theory of Banach spaces,
it comes to no surprise that also their equality cases have
received a lot of attention.
It was first shown by Leichtweiss \cite{1959:Leichtweiss} that only simplices are at Banach--Mazur distance $n$ from the Euclidean ball,
which was later rediscovered in \cite{1992:Palmon}.
An alternative proof
in the more general context of the Gr\"unbaum distance was given in \cite{2011:JimenezNaszodi},
and a different one from comparing radii functionals under affine transformations in \cite{2025:BrandenbergGrundbacher}.
In the symmetric case, it is known (cf.~\cite{2026:GrundbacherKobos}) that the square, or the cube and the cross-polytope are the only distance maximizers in dimensions $n = 2, 3$, respectively, including a stability result for $n = 2$.
In contrast, there are infinitely many distance maximizers for $n \geq 4$, with no clear description of all maximizers available, which complicates obtaining corresponding stability results.
Still, at least a
necessary condition for symmetric convex bodies giving the extremal distance has been obtained by Milman and Wolfson in \cite{1978:MilmanWolfson}.

The literature's apparent lack of stability results in the general case, where the uniqueness of the simplex as extremal body should make them theoretically achievable,
has been to our knowledge
first mentioned by Hug and Schneider \cite{2007:HugSchneider}.
While aiming for estimates in this direction,
they settled for a somewhat weaker result.
Instead of showing that $K$ must be close to the simplex under the assumption $d_{BM}(K,\B_2^n) \geq n - \varepsilon$,
they made the stronger assumption that the so-called \emph{volume quotient} of the volume-extremal ellipsoids approximating $K$ is at least $n - \varepsilon$.
When $\varepsilon \in [0,\varepsilon_0(n))$ for some explicit $\varepsilon_0(n) = O(n^{-19})$, they obtained
\begin{equation}\label{eq:HugSchneider}
    d_{BM}(K,\simplex{n})
    \leq 1 + c_0(n) \varepsilon^{\frac{1}{4}}
\end{equation}
with an explicit $c_0(n) = O(n^{\frac{27}{4}})$.
The first stability estimate for the maximal Banach--Mazur distance to the Euclidean ball
was established by Kobos \cite{2019:Kobos} around a decade later,
showing that any convex body $K \in \CK^n$
with $d_{BM}(K,\B_2^n) \geq n - \varepsilon$ for $\varepsilon \in [0,\varepsilon_0(n))$ satisfies
\begin{equation}\label{eq:Kobos}
    d_{BM}(K,\simplex{n})
    \leq 1 + c_0(n) \varepsilon^{\frac{1}{3}},
\end{equation}
with $\varepsilon_0(n) = O(n^{-8})$ and $c_0(n) = O(n^{\frac{8}{3}})$.
However, \cite{2019:Kobos} left open what the optimal order in $\varepsilon$ for \eqref{eq:Kobos} (or even \eqref{eq:HugSchneider}) should be.

Our second main result improves \eqref{eq:HugSchneider} and \eqref{eq:Kobos} to the level of Theorem~\ref{thm:main},
achieving the optimal linear order and again constants independent of the dimension.

\begin{theorem}\label{thm:DistanceBall}
Let $K\in\CK^n$ with $d_{BM}(K,\B_2^n) \geq n - \varepsilon$ for some $\varepsilon \in (0, \frac{1}{2})$.
Then
\[
    d_{BM}(K,\simplex{n})
    < 1 + 2 \varepsilon + \frac{2 \varepsilon^2}{1-2 \varepsilon}
    = 1 + 2 \varepsilon + \sum_{k = 2}^\infty 2^{k-1} \varepsilon^k.
\]
The resulting linear order of the estimate in $\varepsilon \leq \gamma$ for any $\gamma < \frac{1}{2}$ is optimal.
\end{theorem}

As outlined above, we derive Theorem~\ref{thm:DistanceBall} from Theorem~\ref{thm:main} by providing an appropriate setup that brings the Minkowski asymmetry into the problem.
We do so by establishing
a new relation between the Minkowski asymmetry $s(K)$ and the distance $d_{BM}(K,\B_2^n)$ that interpolates between the known sharp bounds for symmetric and general convex bodies.
It improves inequalities of similar style where different asymmetry measures weaker than $s(K)$ were used (see \cite{2008:BelloniFreund, 2015:BrandenbergKönig}).

\begin{theorem}\label{thm:AsymmetryDistance}
Let $K\in\CK^n$. Then
\[
    d_{BM}(K,\B_2^n)
    \leq \sqrt{n s(K)}.
\]
The inequality is best possible when $s(K)$ is an integer that divides $n$.
\end{theorem}

We verify Theorem~\ref{thm:AsymmetryDistance} through a positive answer to the conjecture of Belloni and Freund \cite[Remark~4]{2008:BelloniFreund} that $K$ contains a translated copy of its volume-minimal circumscribed ellipsoid rescaled by a factor $1/\sqrt{n s(K)}$.
An important difference to the analogous approaches in the literature that established the known sharp bounds for symmetric and general convex bodies is that, like for Theorem~\ref{thm:main}, we again take advantage of all possible translations of the ellipsoids instead of insisting that they should be concentric.
While it is slightly ambiguous whether Belloni and Freund
meant general or concentric
copies of the volume-minimal circumscribed ellipsoid,
Lemma~\ref{lem:befrconj} establishes that the first version is true,
whereas we can show that the second version would be false
(see Remark~\ref{rem:ellipsoid_approx}).
Let us also point out that the analogous result for volume-maximal inscribed ellipsoids is \emph{not} true,
somewhat breaking the usual duality found in the literature
(again see Remark~\ref{rem:ellipsoid_approx}).

Our second application of Theorem~\ref{thm:main} concerns two generalizations of the Rogers--Shephard inequality \cite{1957:RogersShephard}.
In its original form, the inequality states for $K \in \CK^n$ that $\vol_n(DK) \leq \binom{2n}{n} \vol_n(K)$, where $\vol_n$ denotes the $n$-dimensional \cemph{volume} and $DK = \{ x-y : x,y \in K \}$ is the \cemph{difference body} of $K$.
Equality is achieved precisely by simplices.
Schneider \cite{1970:Schneider} generalized this inequality to a higher-order setting,
which laid the foundation for the recent development of the higher-order Brunn--Minkowski theory (see, e.g., \cite{2025:HaddadLangharstLivshytsPutterman,2025+1:HaddadLangharstPuttermanRoysdonYe,2025+2:HaddadLangharstPuttermanRoysdonYe,2025:Langharst,2024:LangharstMarinSolaUlivelli,2024:LangharstXi,2026:ZhouWuYe,2025:ZhouYeZhang}): defining for an integer $m \geq 1$ the $m$-th order difference body of $K \in \CK^n$ as
$D^m K := \{ (z^1, \dots, z^m) \in \cartesian{(\R^n)}{m} : x - z^1, \dots, x - z^m \in K \text{ for some } x \in K \}$,
\cite[Satz~$2$]{1970:Schneider} states that
\begin{equation} \label{eq:Schneider_RS}
    \vol_{nm}(D^m K)
    \leq \binom{n (m+1)}{n} \vol_n(K)^m,
\end{equation}
with equality precisely for simplices.
Here and in the following, we write $\cartesian{X}{m}$ for the \cemph{$m$-fold Cartesian product} of a set $X \subset \R^n$ to properly differentiate from notation like $\CK^n$ and $\simplex{n}$, as well as to clarify that elements of $\cartesian{(\R^n)}{m}$ are always understood as tuples of $m$ vectors of length $n$ instead of as single vectors of length $nm$.

The second generalization in question is for the $L_p$-version of the Rogers--Shephard inequality for bodies $K \in \CK^n$ with $0 \in K$.
Writing $-K := \{-x : x \in K\}$,
Rogers and Shephard \cite[Theorem~$3$]{1958:RogersShephard}
showed that $\vol_n(\conv(K \cup (-K))) \leq 2^n \vol_n(K)$.
This time, equality is attained precisely by simplices with a vertex at $0$.
This naturally raises the question if there is a general $p$-sum version that interpolates between the two known cases.
Partial results in this direction were obtained over the years (cf.~\cite{2008:BianchiniColesanti,2026:FradeliziManuiMeyerNdiaye,2026:ManuiNdiayeZvavitch}),
with the complete answer finally obtained very recently by Kotrbat\'y and Mouamine in the preprint \cite{2026:KotrbatyMouamine}.
With $D_p K := K +_p (-K)$, where $+_p$ denotes the usual $p$-sum for $p \in (1,\infty]$ (see below), \cite[Theorem~$1.2$]{2026:KotrbatyMouamine} states that
\begin{equation} \label{eq:Lp_RS}
    \vol_n(D_p K)
    \leq \sum_{k = 0}^n \binom{n}{k}^2 \binom{n/q}{k/q}^{-1} \vol_n(K).
\end{equation}
Here, $\binom{x}{y} = \frac{\Gamma(1+x)}{\Gamma(1+y) \Gamma(1+x-y)}$ is the generalized binomial coefficient for $x,y \geq 0$ and $q \in [1,\infty)$ satisfies $\frac{1}{p} + \frac{1}{q} = 1$.
Equality is attained again precisely by simplices with a vertex at $0$.
The proof uses a positive answer to the inequality part in the Godbersen conjecture established also in \cite{2026:KotrbatyMouamine}.

We provide stability versions of both \eqref{eq:Schneider_RS} and \eqref{eq:Lp_RS}, extending and improving Böröczky's stability result for the usual Rogers--Shephard inequality \cite{2005:Böröczky}.
To provide a unified framework, we first show a common generalization of \eqref{eq:Schneider_RS} and \eqref{eq:Lp_RS},
whose stability we prove subsequently.
The generalization follows directly from the method for \eqref{eq:Lp_RS} in \cite{2026:KotrbatyMouamine}, with the usual Godbersen conjecture replaced by its higher-order variant from \cite{2025:Kotrbaty} that also follows from the methods in \cite{2026:KotrbatyMouamine}.
We define the \cemph{$m$-th order $L_p$-difference body} of $K \in \CK^n$ with $0 \in K$ as $\diff{p}{m}{K} := \diag{m}{K} +_p (-\cartesian{K}{m})$, where $\diag{m}{K} = \{ (x, \dots, x) \in \cartesian{(\R^n)}{m} : x \in K \}$ is the \cemph{$m$-th diagonal embedding} of $K$.
It is easy to see that $D^m K = \diff{1}{m}{K}$ (cf.~\cite[Proposition~$3.1$]{2025:Kotrbaty}) and $D_p K = \diff{p}{1}{K}$.

\begin{theorem}\label{thm:Lp_RSS}
Let $K \in \CK^n$ with $0 \in K$, $m \geq 1$ an integer, and $p,q \in [1,\infty]$
such that $\frac{1}{p} + \frac{1}{q} = 1$.
Then
\[
    \vol_{nm}(\diff{p}{m}{K})
    \leq \underbrace{\sum_{k = 0}^n \binom{n}{k} \binom{nm}{k} \binom{nm/q}{k/q}^{-1}}_{=:~C_{n,m,q}} \vol_n(K)^m,
\]
with equality if and only if $K$ is a simplex, and $p = 1$ or $0$ is a vertex of $K$.
\end{theorem}

Already Böröczky's proof for the stability of the usual Rogers--Shephard inequality \cite{2005:Böröczky} proceeds by establishing a connection to the Minkowski asymmetry.
In fact, as pointed out in \cite{2005:Böröczky,2026:KotrbatyMouamine}, this connection now follows immediately from the verified inequality in Godbersen's conjecture.
Regarding the closeness to the simplex, this remains true for the theorem above when we use the higher-order version of the Godbersen conjecture.
Therefore, like Theorem~\ref{thm:AsymmetryDistance} does for Theorem~\ref{thm:DistanceBall}, the positive answer to the inequality in Godbersen's conjecture from \cite{2026:KotrbatyMouamine} provides the required setup to use Theorem~\ref{thm:main} for the stability of Theorem~\ref{thm:Lp_RSS}.

Since the equality case for $p > 1$ also includes the positioning of $K$, a complete stability result must necessarily treat this as well.
Our main conceptual contribution is the quantitative control over the positioning of $K$ in two ways.
We show that near-extremal $K$ for Theorem~\ref{thm:Lp_RSS} cannot contain long segments with $0$ as the midpoint,
and that there is an approximating simplex with a vertex close to $0$.
To simplify the statement of the theorem, we define the constants
\begin{align*}
    \gamma_{n,m,q} := \binom{nm/q}{1/q} \frac{C_{n,m,q}}{nm},
        \quad
    c_{n,m,q} := \frac{1}{(n+1)^{2+nm+q} q C_{n,m,q}^2},
        \quad \text{and} \quad
    \xi_{n,m,q} := 2^{2 + \frac{nm}{q}} q C_{n,m,q}^2.
\end{align*}
The rest of the notation used in the theorem is standard and detailed at the end of the introduction.

\begin{theorem}\label{thm:Lp_RSS_stability}
Let $K \in \CK^n$ with $0 \in K$, $m \geq 1$ an integer, and $p,q \in [1,\infty]$
such that
$\frac{1}{p} + \frac{1}{q} = 1$.
If $\vol_{nm}(\diff{p}{m}{K}) \geq C_{n,m,q} (1 - \varepsilon) \vol_n(K)^m$ for some $\varepsilon \in [0,\frac{1}{\gamma_{n,m,q}})$, then
\[
    d_{BM}(K,\simplex{n})
    \leq d_{n,m,q}(\varepsilon)
    := 1 + \gamma_{n,m,q} \cdot \varepsilon + \frac{\gamma_{n,m,q}^2 \cdot \varepsilon^2}{2 (1 - \gamma_{n,m,q} \cdot \varepsilon)}.
\]
Further, if $p > 1$ and $\varepsilon \leq c_{n,m,q}$, then $K$ is placed relative to $0$ such that all $v \in K \cap (-K)$
satisfy
\[
    v
    \in 2 \xi_{n,m,q} \cdot \varepsilon DK,
\]
and there exist a simplex $T$ with (Minkowski) center $c_T$ and a vector $u \in (1-\frac{1}{d_{n,m,q}(\varepsilon)}) T$ such that
\[
    u + \frac{1}{d_{n,m,q}(\varepsilon)} T
    \subset K
    \subset T
        \quad \text{and} \quad
    0 \notin \inte \bigg( c_T + \Big( n - (n+1) \xi_{n,m,q} \cdot \varepsilon \Big) (c_T - T) \bigg).
\]
For all three estimates, the resulting linear order in $\varepsilon \leq c_{n,m,q}$ is optimal.
\end{theorem}

As a third and final application of Theorem~\ref{thm:main}, we obtain improved bounds for
the diameter of the Banach--Mazur compactum.
A direct consequence of the multiplicative triangle inequality and $d_{BM}(K,\B_2^n) \leq n$ for convex bodies $K,L \in \CK^n$
is that $d_{BM}(K,L) \leq n^2$.
However, this estimate is far from optimal.
Very recently, Bizeul and Klartag published the preprint \cite{2025:BizeulKlartag}, improving the previous estimate of Rudelson \cite{2000:Rudelson} of $d_{BM}(K,L) \leq C n ^{\frac{4}{3}} \log(n+1)^9$ for some absolute constant $C>0$ to $d_{BM}(K,L) \leq C n \log(n+1)^4$.
While this (almost) settles the asymptotic behavior of the maximal Banach--Mazur distance in terms of the dimension $n$,
it says little for fixed dimensions owed to the not explicitly stated and possibly large constant $C$.
In low dimensions, the estimate $d_{BM}(K,L)\leq 2.62$ for $n = 2$ due to Brodiuk, Palko, and Prymak \cite{2018:BrodiukPalkoPrymak},
as well as a slight stability improvement of the $n^2$ bound for all $n$ due to Kobos \cite[Corollary~6]{2019:Kobos} seem to be the best known upper bounds.
As a consequence of Theorem~\ref{thm:main}, we obtain the following improvements for low dimensions.

\begin{theorem}\label{thm:BMdistance}
Let $K,L \in \CK^n$.
Then there exists $b(n) \in (n^2 - n + \frac{1}{2},  n^2 - n + \frac{1}{2} + \frac{11}{10n})$ such that
\[
    d_{BM}(K,L)
    \leq b(n).
\]
\begin{table}[ht]
    \centering
    \begin{tabular}{|c|c|c|c|c|c|c|c|}
        \hline
        Dimension & $2$ & $3$ & $4$ & $5$ & $6$ & $7$ & $8$ \\
        \hline 
        Bound $\approx$ & $2.62$ & $6.86$ & $12.76$ & $20.7$ & $30.66$ & $42.63$ & $56.61$ \\
        \hline
    \end{tabular}
    \caption{Numerical values of the current best known upper bounds on the diameter of the Banach--Mazur compactum for $2 \leq n \leq 8$.}
    \label{tab:bn}
\end{table}
\end{theorem}

While an exact formula for the $b(n)$ obtained from our proof could theoretically be extracted,
the resulting expression would be rather unwieldy.
Instead, it is more practical to estimate it and numerically find more precise values of $b(n)$ in the relevant dimensions, which are given in Table~\ref{tab:bn} for $3 \leq n \leq 8$.
The value for $n = 2$ is the one from \cite{2018:BrodiukPalkoPrymak}.
To our knowledge, Theorem~\ref{thm:BMdistance}
provides the best currently known explicit upper bound for any fixed $n \geq 3$.
\pagebreak

The paper is organized as follows.
After the outline, we end the introduction with the general definitions and notations required in the paper.
In Section~\ref{sec:StabilityNearSimplex}, we prove the estimate in Theorem~\ref{thm:main} in the form of several lemmas.
Right after, in Section~\ref{sec:tightness}, we provide examples showing that the estimate in Theorem~\ref{thm:main} cannot be further improved using only our methods, as well as
that the estimate in Theorem~\ref{thm:main} is tight up to the linear order in $\varepsilon$.
We then move on to the Banach--Mazur distance to the Euclidean ball in Section~\ref{sec:StabBall}, verifying Theorems~\ref{thm:DistanceBall}~and~\ref{thm:AsymmetryDistance}.
The results related to the Rogers--Shephard inequality, Theorems~\ref{thm:Lp_RSS}~and~\ref{thm:Lp_RSS_stability}, are shown in Section~\ref{sec:RS}.
Finally, in Section~\ref{sec:DiameterBMcompactum}, we prove Theorem~\ref{thm:BMdistance}.

We close the introduction with the notation used throughout the paper.
Let $X, Y \subset \R^n$, $u,v \in \R^n$, $\Lambda \subset \R$, and $\rho \in \R$.
We denote by $\bd(X)$ and $\inte(X)$ the \cemph{boundary} and the \cemph{interior} of $X$, respectively.
Further, we write $\conv(X)$, $\pos(X)$, $\aff(X)$, and $\lin(X)$ for the \cemph{convex}, \cemph{positive}, \cemph{affine}, and \cemph{linear hull} of $X$, respectively.
The \cemph{Minkowski sum} of $X$ and $Y$ is given by $X + Y := \{ x + y : x \in X, y \in Y \}$ and the \cemph{$v$-translation} of $X$ by $v + X := X + v := X + \{v\}$.
The \cemph{$\rho$-dilation} and the \cemph{$\Lambda$-dilation union} of $X$ are $\rho X := \{ \rho x : x \in X \}$ and $\Lambda X := \bigcup_{\lambda \in \Lambda} \lambda X$, respectively.
For brevity, we write $(-1) X := -X$, $X-Y := X + (-Y)$, and $\Lambda v = \Lambda \{v\}$.
The \cemph{closed line segment} with endpoints $u$ and $v$ is given by $[u,v] := \conv(\{u,v\})$,
where we replace the brackets by parentheses if the respective endpoints should be excluded.
We further define the halfspace $H_{(v,\rho)}^\leq := \{ x \in \R^n : v^T x \leq \rho \}$.
The \cemph{support function} of $X$ evaluated at $v$ is $h_X(v) := \sup \{ v^T x : x \in X \}$,
and its gauge function is given by $\|v\|_X := \inf \{ \mu \geq 0 : v \in \mu X \}$.
For an integer $m \geq 1$, let $[m]:=\{1,2,\dots,m\}$.

For non-empty, compact, convex sets $K,L \subset \R^n$ with $0 \in K \cap L$ and $p \in [1,\infty]$, the \cemph{$p$-sum} of $K$ and $L$ is given by $K +_p L = \{ x \in \R^n : a^T x \leq (h_K(a)^p + h_L(a)^p)^{\frac{1}{p}} \text{ for all } a \in \R^n \}$, where the right-hand side of the defining inequality is understood in the limiting sense for $p = \infty$ as $\max\{h_K(a),h_L(a)\}$.
Equivalently, $K +_p L$ is the unique non-empty, compact, convex subset of $\R^n$ with support function $(h_K^p + h_L^p)^{\frac{1}{p}}$ (cf.~\cite{1962:Firey} and \cite[Chapter~$9$]{2014:Schneider}).

A \cemph{polyhedron} $P \subset \R^n$ is given as $\conv(\{x^1, \dots, x^m \}) + \pos(\{ v^1, \dots, v^k \})$, where $x^i, v^j \in \R^n$ for $i \in [m]$ and $j \in [k]$.
Equivalently, $P = \bigcap_{i \in [r]} H_{(a^i,\beta_i)}^\leq$ for some $a^i \in \R^n \setminus \{0\}$ and $\beta_i \in \R$.
We say that $P$ is a \cemph{polytope} if it is additionally bounded.
A subset $F \subset P$ is a \cemph{face} of $P$ if there exists a closed halfspace $H \subset \R^n$ whose boundary does not intersect the relative interior of $P$ with $F = P \cap H$.
We say that $F$ is a \cemph{facet} of $P$ if $F$ is inclusion-maximal among the faces of $P$ that are not $P$,
and a \cemph{vertex} of $P$ if $F$ is a singleton.
An \cemph{$m$-simplex} is a polytope formed by $m+1$ affinely independent vertices.

The Minkowski asymmetry of $K\in\CK^n$ is more formally given by
\[
    s(K)
    = \inf\{\rho \geq 0: K - c \subset \rho (c-K), c \in \R^n \}.
\]
It is well-known that $s(K) \in [1,n]$,
that $s(K)=1$ if and only if $K$ is symmetric (that is, $K-c = c-K$ for some $c \in \R^n$),
and that $s(K)=n$ if and only if $K$ is an $n$-simplex (see \cite{1963:Grunbaum}).
If $c \in \R^n$ is such that $K-c \subset s(K)(c-K)$,
we say that $c$ is a \cemph{Minkowski center} of $K$.

For $K,L \in \CK^n$, the \cemph{Banach--Mazur distance} between $K$ and $L$ is defined as
\[
    d_{BM}(K,L)
    := \inf\{ \rho\geq 0: u + K \subset AL \subset v + \rho K, A \in \R^{n \times n} \text{ invertible}, u,v \in \R^n\}.
\]
For a general introduction to the Banach--Mazur distance, see \cite{1989:Tomczak-Jaegermann}.

\section{Stability for the Minkowski Asymmetry}\label{sec:StabilityNearSimplex}

The goal of this section is to prove Theorem~\ref{thm:main}.
The starting point is marked by structural observations about convex bodies with large Minkowski asymmetry collected in the following proposition.
Parts (i) and (ii) were obtained by Klee in \cite[($3.10$)]{1953:Klee},
and part (iii) is an extension due to B\"or\"oczky from the discussion and proof for \cite[Proposition~$13$]{2005:Böröczky}.
\pagebreak

\begin{proposition}\label{prop:böröczky}
Let $K \in \CK^n$ with $s(K) > n-1$.
Then
\begin{enumerate}[(i)]
\item $K$ has a unique Minkowski center $c_K \in \inte(K)$.
\item there exists an $n$-simplex $T$ containing $c_K$ whose vertices are common boundary points of $K$ and $s(K) (c_K - K) + c_K$.
In particular, if $T = \conv(\{ v^1, \dots, v^{n+1} \})$, then for any $i \in [n+1]$,
\[
    v^i \in \bd(K)
        \quad \text{and} \quad
    w^i := \frac{s(K) + 1}{s(K)} c_K - \frac{1}{s(K)} v^i \in \bd(K).
\]
\item any $n$-simplex $T$ as in (ii) with (Minkowski) center $c_T \in \inte(T)$ satisfies
\[
    c_K
    \in c_T - \frac{n - s(K)}{s(K) + 1} (T - c_T).
\]
\end{enumerate}
\end{proposition}

We call an $n$-simplex like the set $T$ described in the proposition an \cemph{asymmetry point simplex} of $K$.
Our goal is to show how homothets of such a simplex can be used to cover $K$ to obtain Theorem~\ref{thm:main}.
Before we go into the details of our method, let us briefly discuss an alternative approach.

\begin{remark}
The common boundary points of $K$ and $s(K) (c_K - K) + c_K$ also induce common supporting hyperplanes of the two bodies.
When $K$ is an $n$-simplex, these hyperplanes give back the simplex $s(K) (c_K - K) + c_K$.
It is therefore a natural alternative option to form an approximating simplex from these hyperplanes instead of from the common boundary points.
In essence, this approach is pursued in \cite{2026:BakaevYehudayoff},
where it leads to a slightly weaker bound than the one stated in Theorem~\ref{thm:main}.
Since the main estimate in \cite{2026:BakaevYehudayoff} in this context is sharp (see the remark after the proof there) and Example~\ref{ex:asymmetry_simplex_tightness} shows that also our approach based on asymmetry point simplices cannot be improved, possible further strengthenings of Theorem~\ref{thm:main} might require to involve both types of simplices.
\end{remark}

Let us outline our approach to proving Theorem~\ref{thm:main}.
Asymmetry point simplices already provide an inner approximation by simplices for $K$.
For the outer approximation,
it remains to show how large such a simplex must be scaled to contain a given set up to translation,
and how this can be controlled specifically for $K$ itself.
The former is comparatively easy and taken care of by Lemma~\ref{lem:simplex_circumradius}.
For the latter, we rely on the known boundary points $v^i$ and $w^i$ of $K$ from Proposition~\ref{prop:böröczky} to construct a superset of $K$.
The natural way of doing so results in a description of the superset through the exclusion of certain polyhedra, see Lemma~\ref{lem:set_bound1}.
However, such a description is rather impractical to work with.
The majority of the technical work afterward therefore focuses on the conversion of this set description in terms of exclusion of points
to a description in terms of inclusion of points.
The main observations used for this process throughout the section are outlined once directly relevant.

The following lemma takes care of the rescaling of simplices to cover given sets up to translation.
It is essentially a special case of \cite[Proposition~$6$]{2026:BryantTupper}, though the original result is formally stated only for finite sets $X$.
The extension to compact sets can be taken from the general discussions in the preprint \cite{2026:BryantTupper}, or in the way outlined below.
Let us point out that \cite[Lemma~$12$]{2026:BakaevYehudayoff} also essentially establishes this lemma.

\begin{lemma}\label{lem:simplex_circumradius}
Let $X \subset \R^n$ be non-empty and compact,
and let $S = \bigcap_{i \in [n+1]} H_{(a^i,1)}^\leq$
be an $n$-simplex centered at the origin
for some affinely independent vectors $a^1, \dots, a^{n+1} \in \R^n \setminus \{0\}$.
Then
\[
    R(X,S)
    := \inf \{ \rho \geq 0 : X \subset u + \rho S, u \in \R^n \}
    = \sum_{i \in [n+1]} \frac{h_X(a^i)}{n+1}.
\]
\end{lemma}
\begin{proof}
It is a well-known consequence of Helly's theorem (see, e.g., \cite[Lemma~$2.2$]{2013:BrandenbergKönig}) that there exists some subset $Y \subset X$ with at most $n+1$ elements and $R(Y,S) = R(X,S)$.
We can therefore determine $R(X,S)$ by maximizing $R(Y,S)$ over the finite subsets $Y$ of $X$.

To apply \cite[Proposition~$6$]{2026:BryantTupper}, we note that $\sum_{i \in [n+1]} a^i = 0$.
Indeed, let $x \in \R^n$ be a vertex of $S$.
Then $x$ has inner product $1$ with $n$ of the $a^i$.
Since $S$ has its center at $0$,
the inner product of $x$ with the remaining $a^i$ is $-n$.
It follows that $x^T \sum_{i \in [n+1]} a^i = 0$.
Since this is true for all vertices of $S$ and they linearly span $\R^n$,
we must have $\sum_{i \in [n+1]} a^i = 0$.

We can now apply \cite[Proposition~$6$]{2026:BryantTupper} to find for any finite subset $Y$ of $X$ that
\[
    R(Y,S)
    = \sum_{i \in [n+1]} \frac{h_Y(a^i)}{n+1}
    \leq \sum_{i \in [n+1]} \frac{h_X(a^i)}{n+1}.
\]
Moreover, if for every $i \in [n+1]$ there is some $y^i \in Y$ with $(a^i)^T y^i = h_X(a^i)$, then equality holds.
This establishes the lemma.
\end{proof}

We shall now focus on obtaining a superset of $K$ with a workable description.
Doing so is a fairly long and technical process.
We therefore split it into multiple lemmas over the course of this section, 
throughout which we let everything be defined as in Proposition~\ref{prop:böröczky}.
Moreover, we reuse any notation introduced in those lemmas throughout the rest of the section as well.
Let us already
collect some additional notation and observations:

Since Theorem~\ref{thm:main} is trivially true for $n = 1$, we may suppose that $n \geq 2$.
For brevity, we write $s := s(K)$ and $\varepsilon := n - s \in [0,1)$.
By Proposition~\ref{prop:böröczky}~(iii),
there exist $\lambda_1, \dots, \lambda_{n+1} \geq 0$ with $\sum_{\ell \in [n+1]} \lambda_\ell = \varepsilon$ and
\begin{equation}\label{eq:c_K}
	c_K
	= c_T - \sum_{\ell \in [n+1]} \frac{\lambda_\ell}{s+1} (v^\ell - c_T)
	= \frac{n+1}{s+1} c_T - \sum_{\ell \in [n+1]} \frac{\lambda_\ell}{s+1} v^\ell
	= \sum_{\ell \in [n+1]} \frac{1 - \lambda_\ell}{s+1} v^\ell,
\end{equation}
where we used $(n+1) c_T = \sum_{\ell \in [n+1]} v^\ell$ in the last step.
Since $\frac{n-s}{s+1} < \frac{1}{n}$,
\begin{equation}\label{eq:c_K_int(T)}
	c_K
	\in c_T - \frac{n-s}{s+1} (T - c_T)
	\subset \inte \bigg( c_T - \frac{1}{n} (T - c_T) \bigg)
	\subset \inte(T),
\end{equation}
and consequently for $W := \conv(\{ w^i : i \in [n+1] \})$ also
\begin{equation}\label{eq:c_K_int(W)}
	c_K
		\in c_K + \frac{1}{s} ( c_K - \inte(T) )
		= \inte(W).
\end{equation}
For $i,j \in [n+1]$, we denote by
\[
	m(i,j) \in \argmax \{ \lambda_k : k \in [n+1] \setminus \{i,j\} \}
\]
a choice of index maximizing $\lambda_k$ over $k \in [n+1] \setminus \{i,j\}$.
Moreover, we fix some choices
\[
	m_1 \in \argmax \{ \lambda_k : k \in [n+1] \} \quad \text{and} \quad m_2 \in \argmax \{ \lambda_k : k \in [n+1] \setminus \{m_1\} \}.
\]
Lastly, when $\alpha \in [0, \infty)$ and $\beta \in \R$,
the map $\R \setminus \{1 + \beta\} \ni x \mapsto \frac{\alpha (1 - x)}{1 - x + \beta}$ is differentiable with derivative
\begin{equation}\label{eq:monotone_quotient}
	\frac{\alpha ( (1-x) - (1 - x + \beta) )}{(1 - x + \beta)^2} = \frac{- \alpha \beta}{(1 - x + \beta)^2}
\end{equation}
and is thus increasing for $\beta \leq 0$ and decreasing for $\beta \geq 0$.

The following lemma covers the announced first step of finding a superset of $K$, though in the impractical exclusion form.
All following steps provide set descriptions via inclusion of points either for precisely this set,
or for very slightly larger sets to avoid unnecessary technical difficulties.

\begin{lemma}\label{lem:set_bound1}
We have
\[
	K \subset \R^n \setminus \Bigg( \bigcup_{i \in [n+1]} \big( w^i - \inte(\pos( T - w^i )) \big) \cup \big( v^i + \inte(\pos( T - w^i )) \big) \Bigg).
\]
\end{lemma}
\begin{proof}
Clearly, $T \subset K$ implies $\inte(T) \subset \inte(K)$.
Thus, $w^i \in \bd(K)$ for any $i \in [n+1]$ shows by the convexity of $K$ for any $x \in \inte(T)$ that
\[
	K \cap \big( x + [0,\infty) (w^i - x) \big)
    = [x,w^i].
\]
Now,
any $y \in w^i - \inte(\pos( T - w^i )) = w^i - (0,\infty) (\inte(T) - w^i)$ for $i \in [n+1]$
can be expressed as
\[
	y
    = w^i - \rho (x - w^i)
    = x + (1 + \rho) (w^i - x)
\]
for some $x \in \inte(T)$ and $\rho > 0$ and, hence, does not lie in $K$.
Lastly, we observe that for any $x \in K$, the inclusion $\frac{1}{s} (c_K - K) + c_K \subset K$ shows

\[
	z
    := \frac{s + 1}{s} c_K - \frac{1}{s} x
	= \frac{1}{s} (c_K - x) + c_K
	\in K.
\]
Now, having additionally $x \in v^i + \inte(\pos( T - w^i ))$ for some $i \in [n+1]$ would imply
\[
	z \in \frac{s + 1}{s} c_K - \frac{1}{s} v^i - \frac{1}{s} \inte(\pos( T - w^i ))
		= w^i - \inte(\pos( T - w^i )),
\]
which is not possible by $z \in K$ and the first part of the proof.
\end{proof}

\begin{remark}
Notice that we could use the first argument in the above proof to show also
\begin{equation}\label{eq:additional_cutoff}
	K \cap \big( v^i - \inte(\pos( T - v^i )) \big)
   	= \emptyset
\end{equation}
for any $i \in [n+1]$.
However, we claim that
\begin{equation}\label{eq:smaller_vi_cone}
	(- \pos( T - v^i )) \setminus \{ 0 \}
		= \pos(v^i - T) \setminus \{ 0 \}
		\subset \inte(\pos( T - w^i )),
\end{equation}
which then implies that \eqref{eq:additional_cutoff} does not improve the statement of the above lemma.
Since we require \eqref{eq:smaller_vi_cone} later on,
we provide a proof,
for which it suffices to verify that $v^i - v^j \in \inte(\pos(T - w^i))$ for all $j \in [n+1] \setminus \{ i \}$.

Since $\lambda_i + \lambda_j \leq \varepsilon < 1$,
there exists some $\rho > 0$ with
\[
	\frac{s}{1 - \lambda_j} < \rho
        \quad \text{and} \quad
    \rho \lambda_i < s.
\]
Indeed, this is trivially true if $\lambda_i = 0$,
and otherwise follows from $\frac{s}{1 - \lambda_j} < \frac{s}{\lambda_i}$.
For such a choice of $\rho$, we define
$\rho_i := 1 - \frac{\rho \lambda_i}{s}$,
$\rho_j := \frac{\rho (1 - \lambda_j)}{s} - 1$,
and $\rho_k := \frac{\rho (1 - \lambda_k)}{s}$ for $k \in [n+1] \setminus \{i,j\}$.
The defining properties of $\rho$ yield $\rho_1, \dots, \rho_{n+1} > 0$.
Moreover,
\[
	\sum_{\ell \in [n+1]} \rho_\ell
	= - \frac{\rho \lambda_i}{s} - \frac{\rho (1 - \lambda_i)}{s} + \sum_{\ell \in [n+1]} \frac{\rho (1 - \lambda_\ell)}{s}
	= - \frac{\rho}{s} + \frac{\rho (n + 1 - \varepsilon)}{s}
	= \rho
\]
and, by \eqref{eq:c_K},
\begin{align*}
	\sum_{\ell \in [n+1]} \rho_\ell (v^\ell - w^i)
	& = \bigg( 1 - \frac{\rho \lambda_i}{s} \bigg) v^i - \frac{\rho (1 - \lambda_i)}{s} v^i - v^j + \sum_{\ell \in [n+1]} \bigg( \frac{\rho (1 - \lambda_\ell)}{s} v^\ell - \rho_\ell w^i \bigg)
    \\
	& = v^i - v^j + \rho \bigg( \frac{s+1}{s} c_K - \frac{1}{s} v^i \bigg) - \rho w^i
	= v^i - v^j.
\end{align*}
Since
\[
    \sum_{\ell \in [n+1]} (0,\infty) (v^\ell - w^i)
    = (0,\infty) \inte(T - w^i)
    = \inte(\pos( T - w^i )),
\]
we conclude that $v^i - v^j \in \inte(\pos( T - w^i ))$ and therefore that \eqref{eq:smaller_vi_cone} is true.
\end{remark}

Obtaining the desired interior description of the superset in Lemma~\ref{lem:set_bound1} proceeds in two steps,
one for each type of polyhedra excluded in the lemma.
We start with the first type, where we regroup the hyperplanes inducing the boundaries of the polyhedra to obtain the simplices introduced in the following lemma.
Over the course of the next three lemmas, we show that the union of these simplices is precisely the set arising from excluding all polyhedra of the first type.

\begin{lemma}\label{lem:spikes}
For $i \in [n+1]$, define
\begin{equation}\label{eq:z^i}
	z^i
	:= v^i + \frac{1}{1 - \varepsilon} \Bigg( \varepsilon v^i - \sum_{\ell \in [n+1]} \lambda_\ell v^\ell \Bigg)
	\in v^i + \frac{\varepsilon}{1 - \varepsilon} (v^i - T)
	\subset v^i + \pos(v^i - T)
\end{equation}
and $T^i := \conv(\{ z^i, v^j : j \in [n+1] \setminus \{i\} \})$.
Then
\begin{enumerate}[(i)]
\item $T^i$ is an $n$-simplex with $T \subset T^i$ and $c_K \in \inte(T^i)$.
\item for $j \in [n+1] \setminus \{i\}$, $w^j$ lies in the relative interior of a facet $F^i_j$ of $T^i$, where $F^i_j$ is induced by the hyperplane
$H^i_j := w^j + \lin(\{ w^j-v^k : k \in [n+1] \setminus \{i,j\} \})$.
\end{enumerate}
\end{lemma}
\begin{proof}
By $z^i \in v^i + \pos(v^i - T) = v^i + \pos(\{ v^i - v^j : j \in [n+1] \setminus \{i\} \})$,
there exists some $y^i \in \conv(\{ v^j : j \in [n+1] \setminus \{i\} \})$ with $v^i \in [y^i,z^i] \subset T^i$.
This already verifies $T \subset T^i$ and thus, by \eqref{eq:c_K_int(T)}, also $c_K \in \inte(T^i)$.
Since $T$ has non-empty interior, $T^i$ is a convex hull of $n+1$ points in $\R^n$ with non-empty interior.
It must therefore be an $n$-simplex, so (i) follows.

Now, for any $j \in [n+1] \setminus \{i\}$,
\eqref{eq:c_K} shows
\[
	\frac{1 - \varepsilon}{s} z^i + \sum_{k \in [n+1] \setminus \{i,j\}} \frac{1}{s} v^k
	= \Bigg( \sum_{\ell \in [n+1]} \frac{1 - \lambda_\ell}{s} v^\ell \Bigg) - \frac{1}{s} v^j
	= \frac{s+1}{s} c_K - \frac{1}{s} v^j
	= w^j.
\]
Since $\frac{1 - \varepsilon}{s} + \frac{n-1}{s} = 1$ and $\varepsilon < 1$,
we conclude that $w^j$ lies in the relative interior of the facet
$F^i_j = \conv(\{ z^i, v^k : k \in [n+1] \setminus \{i,j\} \})$ of $T^i$.
Thus, it also immediately follows $\aff(F^i_j) = H^i_j$.
\end{proof}

Let us outline the main idea to verify that the $T^i$ from above cover what is left after excluding the first type of polyhedra in Lemma~\ref{lem:set_bound1}.
For any ray starting from $c_K$,
we show that it stays at least as long in the union of the $T^i$ as it stays in the region not excluded by the polyhedra.
This is achieved by considering where the rays meet $\bd(W)$ to decide which of the $T^i$ we want to show membership in,
namely the $T^i$ containing all $w^j$ spanning the facet of $W$ that the ray meets.
The following technical lemma lets us determine how long the ray stays within $T^i$.
This approach can be seen to be the appropriate one by noticing that the lemma shows in particular that the ray meets the boundaries of two of the $T^i$ in the same point whenever it passes through a face of $W$ spanned by any $n-1$ of the $w^j$.
In this sense, the facial structure of $W$ already determines how the different $T^i$ are related (see also Remark~\ref{rem:exact_Ti_description}).

\begin{lemma}\label{lem:lines_meeting_spikes}
Let $i \in [n+1]$, choose some $\sigma_k \in \R$ for all $k \in [n+1] \setminus \{i\}$ with $\sum_{k \in [n+1] \setminus \{i\}} \sigma_k = 1$,
and define the line
\[
	L
	:= c_K + \lin \Bigg( \Bigg\{ \sum_{k \in [n+1] \setminus \{i\}} \sigma_k (w^k - c_K) \Bigg\} \Bigg).
\]
Then for any $j \in [n+1] \setminus \{i\}$, the line $L$
intersects the hyperplane $H^{i}_j$ if and only if $\sigma_j \neq \frac{1}{s+1}$.
In this case, $L \cap H^i_j$ contains precisely the point
\[
	c_K + \frac{s}{(s+1) \sigma_j - 1} \sum_{k \in [n+1] \setminus \{i\}} \sigma_k (w^k - c_K)
		= w^j + \sum_{k \in [n+1] \setminus \{i,j\}} \frac{\sigma_k}{(s+1) \sigma_j - 1} (w^j - v^k).
\]
\end{lemma}
\begin{proof}
By \eqref{eq:c_K_int(W)}, $c_K$ is an interior point of $W$,
whereas $\sum_{k \in [n+1] \setminus \{i\}} \sigma_k w^k$ lies in the affine hull of one of its facets.
Thus, $L$ is indeed a line and not just a singleton.
Since $c_K \in \inte(T^i)$ by Lemma~\ref{lem:spikes}~(i) and $H^i_j$ induces a facet of $T^i$ by Lemma~\ref{lem:spikes}~(ii),
we have $c_K \notin H^i_j$.
Therefore, the line $L$ intersects the hyperplane $H^i_j$ in at most one point.

First, assume that $\sigma_j = \frac{1}{s+1}$.
We clearly have $w^j \in (L - c_K + w^j) \cap H^i_j$.
Moreover, by $\sum_{k \in [n+1] \setminus \{i,j\}} \sigma_k = 1 - \sigma_j = \frac{s}{s+1}$,
\begin{align*}
    c_K + \Bigg(\sum_{k \in [n+1] \setminus \{i\}} \sigma_k (w^k - c_K) \Bigg) - c_K + w^j
    & = w^j + \frac{w^j - c_K}{s+1} + \sum_{k \in [n+1] \setminus \{i,j\}} \frac{\sigma_k}{s} (c_K - v^k)
    \\
    & = w^j + \sum_{k \in [n+1] \setminus \{i,j\}} \frac{\sigma_k}{s} (w^j - v^k)
\end{align*}
is a point different from $w^j$ in $(L - c_K + w^j) \cap H^i_j$.
Thus, the line $L - c_K + w^j$ is a subset of the hyperplane $H^i_j$.
Consequently, if $L$ intersected $H^i_j$ in this case, we would obtain $L \subset H^i_j$, which is a contradiction.

Now, assume $\sigma_j \neq \frac{1}{s+1}$.
Using $s (w^k - c_K) = c_K - v^k$ for any $k \in [n+1]$ and $1 - \sigma_j = \sum_{k \in [n+1] \setminus \{i,j\}} \sigma_k$,
we compute

\begin{align*}
	& c_K + \frac{s}{(s+1) \sigma_j - 1} \sum_{k \in [n+1] \setminus \{i\}} \sigma_k (w^k - c_K)
	\\
	& = c_K + \frac{s \sigma_j}{(s+1) \sigma_j - 1} (w^j - c_K)
		+ \sum_{k \in [n+1] \setminus \{i,j\}} \frac{\sigma_k}{(s+1) \sigma_j - 1} (c_K - v^k)
	\\
	& = w^j + \frac{1 - \sigma_j}{(s+1) \sigma_j - 1} (w^j - c_K)
		+ \sum_{k \in [n+1] \setminus \{i,j\}} \frac{\sigma_k}{(s+1) \sigma_j - 1} (c_K - v^k)
	\\
	& = w^j + \sum_{k \in [n+1] \setminus \{i,j\}} \frac{\sigma_k}{(s+1) \sigma_j - 1} (w^j - c_K)
		+ \sum_{k \in [n+1] \setminus \{i,j\}} \frac{\sigma_k}{(s+1) \sigma_j - 1} (c_K - v^k)
	\\
	& = w^j + \sum_{k \in [n+1] \setminus \{i,j\}} \frac{\sigma_k}{(s+1) \sigma_j - 1} (w^j - v^k).
\end{align*}
This verifies the claimed equality of points in the lemma.
It is clear that the left-hand description of the point represents a point in $L$,
whereas the right-hand one represents a point in $H^i_j$.
Hence, the point must be the unique point in $L \cap H^i_j$ in this case.
\end{proof}

\begin{lemma}\label{lem:set_bound2}
We have
\[
	\R^n \setminus \bigcup_{i \in [n+1]} \big( w^i - \inte(\pos( T - w^i )) \big)
	\subset \bigcup_{i \in [n+1]} T^i.
\]
\end{lemma}
\begin{proof}
Let $x \in \R^n$ be chosen arbitrarily.
If $x = c_K$, then $c_K \in K$ together with Lemma~\ref{lem:set_bound1} shows that $x$ is contained in the left-hand set,
whereas Lemma~\ref{lem:spikes} shows that $x$ is contained in the right-hand set.

From now on, assume $x \neq c_K$.
By \eqref{eq:c_K_int(W)},
the ray $R := c_K + [0,\infty) (x - c_K)$ intersects the boundary of $W$ in a unique point $y$.
Since $W$ is an $n$-simplex, there exist some $i \in [n+1]$ and for all $k \in [n+1] \setminus \{i\}$ some $\sigma_k \geq 0$
such that $\sum_{k \in [n+1]} \sigma_k = 1$ and $y = \sum_{k \in [n+1] \setminus \{i\}} \sigma_k w^k$.
By Lemma~\ref{lem:spikes}~(i),
the ray $R$ also intersects the boundary of the simplex $T^i$ in a unique point $z$.
To determine $z$,
it suffices to determine which of the hyperplanes that induce the facets of $T^i$ are intersected by $R$
and then among the respective intersection points choose the one closest to $c_K$.
To this end, we first observe that $c_K$ lies strictly between the parallel hyperplanes $\aff(\{ v^k : k \in [n+1] \setminus \{i\} \})$
and $\aff(\{ w^k : k \in [n+1] \setminus \{i\} \}) = c_K + \frac{1}{s} (c_K - \aff(\{ v^k : k \in [n+1] \setminus \{i\} \}))$.
Since $R$ intersects the latter in $y$, it cannot also intersect the former.
Hence, Lemma~\ref{lem:spikes}~(ii) shows that $z$ lies in one of the hyperplanes $H^i_j$ for some $j \in [n+1] \setminus \{i\}$.

Lemma~\ref{lem:lines_meeting_spikes} now shows for $j \in [n+1] \setminus \{i\}$
that the line $\aff(R) = c_K + \lin(\{y-c_K\})$ intersects the hyperplane $H^i_j$ if and only if $\sigma_j \neq \frac{1}{s+1}$.
In this case, their unique intersection point is
\begin{equation}\label{eq:intersection_R_Hij}
	c_K + \frac{s}{(s+1) \sigma_j - 1} (y - c_K).
\end{equation}
This point lies in $R = c_K + [0,\infty) (y - c_K)$ if and only if $\sigma_j > \frac{1}{s+1}$,
so $R$ intersects $H^i_j$ for precisely those $j \in [n+1] \setminus \{i\}$ with $\sigma_j > \frac{1}{s+1}$.
Moreover, we see that the point in \eqref{eq:intersection_R_Hij} lies closer to $c_K$ the larger $\sigma_j$ is.
Altogether, if $j \in \argmax \{ \sigma_k : k \in [n+1] \setminus \{i\} \}$,
which by $\sum_{k \in [n+1] \setminus \{i\}} \sigma_k = 1$ implies $\sigma_j \geq \frac{1}{n} > \frac{1}{s+1}$,
then
\[
	z
	= c_K + \frac{s}{(s+1) \sigma_j - 1} (y - c_K)
	\in \bd(T^i).
\]
By Lemma~\ref{lem:lines_meeting_spikes}, we also have
\[
	z
	= w^j + \sum_{k \in [n+1] \setminus \{i,j\}} \frac{\sigma_k}{(s+1) \sigma_j - 1} (w^j - v^k),
\]
so in particular $z \in w^j - \pos(T - w^j)$.
Hence,
\eqref{eq:c_K_int(T)} shows for any $\rho > 0$ that
\begin{align*}
	c_K + (1+\rho) (z - c_K)
	& = w^j + (1 + \rho) (z - w^j) - \rho (c_K - w^j)
    \\
	& \in w^j - \pos(T - w^j) - \inte(\pos( T - w^j ))
	= w^j - \inte(\pos( T - w^j )).
\end{align*}
Altogether, if $x \in R = c_K + [0,\infty) (z - c_K)$ is contained in the left-hand set in the lemma,
then $x \in c_K + [0,1] (z - c_K) = [c_K,z] \subset T^i$,
where $i$ was determined by where $R$ meets $\bd(W)$.
\end{proof}

\begin{remark}\label{rem:exact_Ti_description}
In the above lemma, also equality of the considered sets can be shown.
Since we do not need this fact later, we only give an outline how to obtain the reverse inclusion.

First, we note that $z$ is the furthest point from $c_K$ on the ray $R$ that still lies in the union of the $T^\ell$.
Indeed, fix some $\ell \in [n+1] \setminus \{i\}$.
If $R$ intersects $\conv(\{ w^k : k \in [n+1] \setminus \{i,\ell\} \})$, then it follows from Lemma~\ref{lem:lines_meeting_spikes} that $R$ meets $\bd(T^i)$ and $\bd(T^\ell)$ in the same point.
Otherwise, $R$ points into the halfspace that is bounded by the hyperplane $H := \aff(\{c_K, w^k : k \in [n+1] \setminus \{i,\ell\} \})$ and contains $w^j$.
The part of $T^\ell$ in that same halfspace is the convex hull of the $v^k$ with $k \in [n+1] \setminus \{\ell\}$ and the point $p$ in $H \cap [v^i,z^\ell]$.
By the above, $p$ also lies in the boundary of $T^i$, so that $T^i$ contains the part of $T^\ell$ in the relevant halfspace.

It remains to show that $z$ lies in the left-hand set in the lemma,
i.e., that $z$ does not lie in an open polyhedron of the form $w^\ell - \inte(\pos( T - w^\ell ))$.
For $\ell \neq i$,
it can be shown that the hyperplane $H^i_\ell$ induces a facet of the closed polyhedron $w^\ell - \pos( T - w^\ell )$,
while it also induces a facet of $T^i$.
Therefore, $H^i_\ell$ strictly separates $c_K \in T^i$ from $w^\ell - \inte(\pos( T - w^\ell ))$.
Consequently, $R$ must meet $H^i_\ell$ before it can enter $w^\ell - \inte(\pos( T - w^\ell ))$.
Since Lemma~\ref{lem:lines_meeting_spikes} and $\sigma_j \geq \sigma_\ell$ show
$\aff(\{c_K,z\}) \cap H^i_\ell \subset \aff(\{c_K,z\}) \setminus [c_K,z)$,
this means $z \notin w^\ell - \inte(\pos( T - w^\ell ))$.
Finally, for $\ell = i$, it can be shown that $w^i - \inte(\pos( T - w^i ))$ is contained
in the halfspace that is bounded by $w^i + \aff(\{ v^k : k \in [n+1] \setminus \{i\} \})$ and does not contain $c_K$.
Hence, the ray $c_K + [0,\infty) (z - c_K)$ does not meet this set by the same argument
why it does not meet the hyperplane $\aff(\{ v^k : k \in [n+1] \setminus \{i\} \})$.
Altogether, $z$ also does not lie in $w^i - \inte(\pos( T - w^i ))$ and is therefore contained in the left-hand set in the lemma.
\end{remark}

Let us point out that if we were to stop here and use the union of the $T^i$ for the superset of $K$,
then we would achieve the upper bound $\frac{1}{1-\varepsilon}$
(which coincides precisely with the one obtained in \cite{2026:BakaevYehudayoff}).
In fact, it is not difficult to show that the convex hull of the union of the $T^i$ is a translated copy of $T$ rescaled by a factor $\frac{1}{1-\varepsilon}$.
However, we still have the second type of polyhedra in Lemma~\ref{lem:set_bound1} that we can exclude to improve the estimate.
To do so, we need to find which parts of the $T^i$ near the $z^i$ are cut off by those polyhedra.
This can be determined with the help of the following technical lemma.

\begin{lemma}\label{lem:spikes_cutoff}
Let $i,j,h \in [n+1]$ be pairwise different.
Then the segment $S := [v^j, z^i]$ intersects the hyperplane
$H := H^h_i - w^i + v^i = v^i + \lin(\{ v^k - w^i : k \in [n+1] \setminus \{i,h\} \})$
in the point
\begin{align*}
	p^i_{j,h}
	& := \frac{\varepsilon (1 - \lambda_h) - \lambda_i}{1 - \lambda_h - \lambda_i} v^j + \frac{(1 - \varepsilon) (1 - \lambda_h)}{1 - \lambda_h - \lambda_i} z^i
    \\
	& = v^i + \frac{\lambda_h (1 - \varepsilon) + \varepsilon - \lambda_j - \lambda_i}{1 - \lambda_h - \lambda_i} (v^j - w^i) + \sum_{k \in [n+1] \setminus \{i,j,h\}} \frac{\lambda_h - \lambda_k}{1 - \lambda_h - \lambda_i} (v^k - w^i)
\end{align*}
\end{lemma} \begin{proof}
We define
\[
	\lambda
    := 1 - \frac{(1 - \varepsilon) (1 - \lambda_h)}{1 - \lambda_h - \lambda_i}
    = \frac{\varepsilon (1 - \lambda_h) - \lambda_i}{1 - \lambda_h - \lambda_i},
\]
which by $0 \leq \lambda_h \leq \lambda_h + \lambda_i \leq \varepsilon < 1$ satisfies
\[
	1
    \geq \lambda
    \geq 1 - \frac{1 - \varepsilon}{1 - \lambda_h - \lambda_i}
    \geq 1 - \frac{1 - \varepsilon}{1 - \varepsilon}
    = 0.
\]
Moreover,
\[
	\frac{\lambda_h (1 - \varepsilon) + \varepsilon - \lambda_j - \lambda_i}{1 - \lambda_h - \lambda_i}
    = \frac{\lambda_h - \lambda_j}{1 - \lambda_h - \lambda_i} + \lambda,
\]
and, by \eqref{eq:c_K},
\begin{align*}
	\sum_{\ell \in [n+1]} (\lambda_h - \lambda_\ell) (v^\ell - w^i)
    & = (\varepsilon - (n+1) \lambda_h) w^i + \sum_{\ell \in [n+1]} ((1 - \lambda_\ell) \lambda_h + (\lambda_h - 1) \lambda_\ell) v^\ell
    \\
	& = (\varepsilon - (1 + \varepsilon) \lambda_h - \lambda_h s) w^i + \lambda_h (s+1) c_K - (1 - \lambda_h) \sum_{\ell \in [n+1]} \lambda_\ell v^\ell
    \\
	& = (\varepsilon - (1 + \varepsilon) \lambda_h) w^i + \lambda_h v^i - (1 - \lambda_h) \sum_{\ell \in [n+1]} \lambda_\ell v^\ell.
\end{align*}
Thus,
\begin{align*}
	& v^i + \frac{\lambda_h (1 - \varepsilon) + \varepsilon - \lambda_j - \lambda_i}{1 - \lambda_h - \lambda_i} (v^j - w^i) + \sum_{k \in [n+1] \setminus \{i,j,h\}} \frac{\lambda_h - \lambda_k}{1 - \lambda_h - \lambda_i} (v^k - w^i)
    \\
	& = v^i + \lambda (v^j - w^i) + \sum_{k \in [n+1] \setminus \{i\}} \frac{\lambda_h - \lambda_k}{1 - \lambda_h - \lambda_i} (v^k - w^i)
    \\
	& = \lambda v^j + v^i - \lambda w^i + \frac{(\varepsilon - (1 + \varepsilon) \lambda_h) w^i + \lambda_h v^i - (\lambda_h - \lambda_i) (v^i - w^i)  - (1 - \lambda_h) \sum_{\ell \in [n+1]} \lambda_\ell v^\ell}{1 - \lambda_h - \lambda_i}
    \\
	& = \lambda v^j + \frac{(1 - \lambda_h) v^i - (1 - \lambda_h) \sum_{\ell \in [n+1]} \lambda_\ell v^\ell}{1 - \lambda_h - \lambda_i}
	= \lambda v^j + (1 - \lambda) z^i
\end{align*}
shows that both descriptions of $p^i_{j,h}$ given in the lemma do indeed coincide.
Finally, $\lambda \in [0,1]$ and the first description of $p^i_{j,h}$ show $p^i_{j,h} \in S$,
whereas the second description shows $p^i_{j,h} \in H$.
\end{proof}

\begin{lemma}\label{lem:set_bound3}
We have
\[
	K
    \subset \bigcup_{i \in [n+1]} \conv \big( \big\{ v^j, q^i_j : j \in [n+1] \setminus \{i\} \big\} \big),
\]
where for $i,j \in [n+1]$, $i \neq j$,
\[
	q^i_j
    := \frac{\varepsilon (1 - \lambda_{m(i,j)}) - \lambda_i}{1 - \lambda_{m(i,j)} - \lambda_i} v^j + \frac{(1 - \varepsilon) (1 - \lambda_{m(i,j)})}{1 - \lambda_{m(i,j)} - \lambda_i} \Bigg( v^i + \frac{1}{1 - \varepsilon} \Bigg( \varepsilon v^i - \sum_{\ell \in [n+1]} \lambda_\ell v^\ell \Bigg) \Bigg).
\]
\end{lemma}
\begin{proof}
By Lemmas~\ref{lem:set_bound1}~and~\ref{lem:set_bound2},
it suffices to prove for $i \in [n+1]$ that
\begin{equation}\label{eq:spikes_cutoff}
	R^i
		:= T^i \setminus ( v^i + \inte(\pos( T - w^i )) )
		\subset \conv(\{ v^j, q^i_j : j \in [n+1] \setminus \{i\} \}).
\end{equation}
If $\lambda_i = \varepsilon$, then $z^i = v^i$ and thus $T^i = T$.
Moreover, $\lambda_{m(i,j)} = 0$ and thus $q^i_j = v^i = z^i$ for any $j \in [n+1] \setminus \{i\}$.
Hence, \eqref{eq:spikes_cutoff} trivially holds in this case.

From now on, we assume $\lambda_i < \varepsilon$.
Then \eqref{eq:smaller_vi_cone} and \eqref{eq:z^i}
show $z^i \in v^i + \inte(\pos( T - w^i ))$.
Moreover, we have for any $j \in [n+1] \setminus \{i\}$ that
$\lambda_{m(i,j)} (1 - \varepsilon) + \varepsilon - \lambda_j - \lambda_i \geq 0$
and for any $k \in [n+1] \setminus \{i,j,m(i,j)\}$ that $\lambda_{m(i,j)} \geq \lambda_k$.
Thus, Lemma~\ref{lem:spikes_cutoff} shows $q^i_j = p^i_{j,m(i,j)} \in v^i + \pos( T - w^i )$.
Altogether, with $Q^i := \conv(\{ q^i_j : j \in [n+1] \setminus \{i\} \})$ it follows that
\begin{equation}\label{eq:C^i}
	C^i
    := \conv(\{ z^i, q^i_j : j \in [n+1] \setminus \{i\} \}) \setminus Q^i
	\subset  v^i + \inte(\pos( T - w^i )).
\end{equation}

Now, let $x \in R^i \subset T^i$ be chosen arbitrarily.
Then there exists $y \in \conv(\{ v^j : j \in [n+1] \setminus \{i\} \})$ such that $x \in [y,z^i]$.
For any $j \in [n+1] \setminus \{i\}$,
the discussion regarding \eqref{eq:monotone_quotient}, $\lambda_{m(i,j)} \leq \varepsilon - \lambda_i$, and $\lambda_i < \varepsilon$ show
\[
	\frac{(1 - \varepsilon) (1 - \lambda_{m(i,j)})}{1 - \lambda_{m(i,j)} - \lambda_i}
	\leq \frac{(1 - \varepsilon) (1 - (\varepsilon - \lambda_i))}{1 - (\varepsilon - \lambda_i) - \lambda_i}
	< 1,
\]
so Lemma~\ref{lem:spikes_cutoff} yields $q^i_j \in [v^j,z^i)$.
For $\rho > 0$ sufficiently large, it follows that
\[
	v^j
    \in z^i + \rho (\conv(\{ z^i, q^i_j : j \in [n+1] \setminus \{i\} \}) - z^i)
\]
and consequently for all $\rho > 0$ sufficiently large also
\[
	x
	\in T^i
	\subset z^i + \rho (\conv(\{ z^i, q^i_j : j \in [n+1] \setminus \{i\} \}) - z^i).
\]
Now, $x \in R^i \subset T^i \setminus C^i$ by \eqref{eq:C^i} yields $x \neq z^i$ and thus the existence of a smallest $\rho > 0$ with
\[
	x - z^i
		\in \rho (\conv(\{ z^i, q^i_j : j \in [n+1] \setminus \{i\} \}) - z^i)
		= \rho (\conv(\{ 0, q^i_j - z^i : j \in [n+1] \setminus \{i\} \}).
\]
Since $x \notin C^i$, this $\rho$ satisfies $\rho \geq 1$.
By the minimality of $\rho$, there must exist some $\sigma_j \geq 0$ for all $j \in [n+1] \setminus \{i\}$ with
$\sum_{j \in [n+1] \setminus \{i\}} \sigma_j = \rho$ and
\[
	x - z^i
    = \sum_{j \in [n+1] \setminus \{i\}} \sigma_j (q^i_j - z^i).
\]
Therefore,
\[
	z
    := z^i + \frac{1}{\rho} (x - z^i)
	\in z^i + \conv(\{ q^i_j - z^i : j \in [n+1] \setminus \{i\} \})
	= Q^i,
\]
and by $\rho \geq 1$ also $z \in [x,z^i]$.
Finally, the choice of $y$ and $z \in Q^i$ yield
\[
	x
	\in [y,z]
	\subset \conv(\{ v^j, q^i_j : j \in [n+1] \setminus \{i\} \}).
\]
Since $x \in R^i$ has been chosen arbitrarily, \eqref{eq:spikes_cutoff} and thus the lemma follow.
\end{proof}

With the above lemma,
the process of finding a sufficiently small, workable superset of $K$ is finally complete.
We can now end this section with the proof of the stability result in Theorem~\ref{thm:main}.

\begin{proof}[Proof of the upper bound in Theorem~\ref{thm:main}]
We assume $n \geq 2$ without loss of generality.
By Proposition~\ref{prop:böröczky}, there exists an asymmetry point simplex $T = \conv(\{ v^1, \dots, v^{n+1} \})$ of $K$. Applying an appropriate translation if necessary,
we may assume $T$ to be centered at the origin.
In this case, there exist affinely independent vectors $a^1, \dots, a^{n+1} \in \R^n \setminus \{0\}$ such that $T = \bigcap_{i \in [n+1]} H_{(a^i,1)}^\leq$
and for all $i,j \in [n+1]$, $i \neq j$, we have $(a^i)^T v^j = 1$ and $(a^i)^T v^i = -n$.

With the notation introduced in this section,
Lemma~\ref{lem:set_bound3} shows
\[
	K
	\subset L
	:= \conv(\{ v^i, q^i_j : i,j \in [n+1], i \neq j \})
\]
with
\[
	q^i_j = \frac{\varepsilon (1 - \lambda_{m(i,j)}) - \lambda_i}{1 - \lambda_{m(i,j)} - \lambda_i} v^j + \frac{(1 - \varepsilon) (1 - \lambda_{m(i,j)})}{1 - \lambda_{m(i,j)} - \lambda_i} \Bigg( v^i + \frac{1}{1 - \varepsilon} \Bigg( \varepsilon v^i - \sum_{\ell \in [n+1]} \lambda_\ell v^\ell \Bigg) \Bigg).
\]
By Lemma~\ref{lem:simplex_circumradius},
the set $L$, and thus also $K$, is contained in a translated copy of $T$ rescaled by a factor $\sum_{i \in [n+1]} \frac{h_L(a^i)}{n+1}$.
Therefore, the claimed upper bound on $d_{BM}(K,\simplex{n})$ can be proved by showing
\begin{equation}\label{eq:support_ineq}
	\sum_{i \in [n+1]} \frac{h_L(a^i)}{n+1}
		\leq 1 + \varepsilon + \frac{\varepsilon^2}{2 (1 - \varepsilon)}.
\end{equation}

To this end, let $i,j,k \in [n+1]$, $j \neq k$.
Then $(a^i)^T v^k \leq 1$,
and \eqref{eq:z^i} shows $(a^i)^T z^i \leq (a^i)^T v^i \leq -n$.
Thus, by $q^k_j \in [v^j,z^k]$, for $k = i$ we have
\[
	(a^i)^T q^k_j
    = (a^i)^T q^i_j
    \leq \max \{ (a^i)^T v^j, (a^i)^T z^i \}
    = 1.
\]
Furthermore, with $x := \sum_{\ell \in [n+1]} \lambda_\ell v^\ell \in \varepsilon T$, we obtain for $k \neq i$ that
\begin{align*}
	(a^i)^T q^k_j
	& \leq \frac{\varepsilon (1 - \lambda_{m(k,j)}) - \lambda_k}{1 - \lambda_{m(k,j)} - \lambda_k} + \frac{(1 - \varepsilon) (1 - \lambda_{m(k,j)})}{1 - \lambda_{m(k,j)} - \lambda_k} + \frac{1 - \lambda_{m(k,j)}}{1 - \lambda_{m(k,j)} - \lambda_k} ( \varepsilon - (a^i)^T x )
    \\
	& = 1 + \frac{1 - \lambda_{m(k,j)}}{1 - \lambda_{m(k,j)} - \lambda_k} ( \varepsilon - (a^i)^T x ).
\end{align*}
Now, $x \in \varepsilon T$ implies $\varepsilon - (a^i)^T x \geq 0$,
so altogether by the discussion regarding \eqref{eq:monotone_quotient},

\begin{align*}
	h_L(a^i)
	& \leq \max_{k \in [n+1] \setminus \{i\}} \Bigg( 1 + \frac{1 - \lambda_{m(k,k)}}{1 - \lambda_{m(k,k)} - \lambda_k} ( \varepsilon - (a^i)^T x ) \Bigg)
    \\
	& = \begin{cases}
		1 + \frac{1 - \lambda_{m_1}}{1 - \lambda_{m_1} - \lambda_{m_2}} ( \varepsilon - (a^i)^T x ), & \text{ if } i = m_1,
        \vspace{5pt} \\
		1 + \frac{1 - \lambda_{m_2}}{1 - \lambda_{m_2} - \lambda_{m_1}} ( \varepsilon - (a^i)^T x ), & \text{ if } i \neq m_1.
	\end{cases}
\end{align*}
In the proof of Lemma~\ref{lem:simplex_circumradius}, we have shown $\sum_{i \in [n+1]} a^i = 0$. Together with
\[
    \varepsilon - (a^{m_1})^T x = \varepsilon + n \lambda_{m_1} - \sum_{i \in [n+1] \setminus \{m_1\}} \lambda_i = (n+1) \lambda_{m_1},
\]
we obtain
\begin{align*}
	\sum_{i \in [n+1]} \frac{h_L(a^i)}{n+1}
	& \leq 1 + \frac{(\lambda_{m_2} - \lambda_{m_1}) (\varepsilon - (a^{m_1})^T x) + (1 - \lambda_{m_2}) \sum_{i \in [n+1]} (\varepsilon - (a^i)^T x)}{(n+1) (1 - \lambda_{m_2} - \lambda_{m_1})}
    \\
	& = 1 + \frac{(\lambda_{m_2} - \lambda_{m_1}) (n+1) \lambda_{m_1} + (1 - \lambda_{m_2}) (n+1) \varepsilon}{(n+1) (1 - \lambda_{m_2} - \lambda_{m_1})}
    \\
	& = 1 + \frac{(\lambda_{m_2} - \lambda_{m_1}) \lambda_{m_1} + (1 - \lambda_{m_2}) \varepsilon}{1 - \lambda_{m_2} - \lambda_{m_1}}
	= 1 + \frac{\varepsilon + \lambda_{m_2} (\lambda_{m_1} - \varepsilon) - \lambda_{m_1}^2}{1 - \lambda_{m_2} - \lambda_{m_1}}.
\end{align*}
Now, observe that the map
$\R \setminus \{1 - \lambda_{m_1}\} \ni x \mapsto 1 + \frac{\varepsilon + x (\lambda_{m_1} - \varepsilon) - \lambda_{m_1}^2}{1 - x - \lambda_{m_1}}$
is differentiable with derivative
\[
	\frac{(\lambda_{m_1} - \varepsilon) (1 - x - \lambda_{m_1}) + \varepsilon + x (\lambda_{m_1} - \varepsilon) - \lambda_{m_1}^2}
			{(1 - x - \lambda_{m_1})^2}
		= \frac{\lambda_{m_1} (1 + \varepsilon - 2 \lambda_{m_1})}{(1 - x - \lambda_{m_1})^2}
		\geq 0.
\]
Therefore, by $\lambda_{m_2} \leq \varepsilon - \lambda_{m_1}$,
\[
	1 + \frac{\varepsilon + \lambda_{m_2} (\lambda_{m_1} - \varepsilon) - \lambda_{m_1}^2}{1 - \lambda_{m_2} - \lambda_{m_1}}
		\leq 1 + \frac{\varepsilon - (\lambda_{m_1} - \varepsilon)^2 - \lambda_{m_1}^2}{1 - \varepsilon}
		= 1 + \varepsilon + \frac{2}{1 - \varepsilon} \lambda_{m_1} (\varepsilon - \lambda_{m_1}).
\]
Finally, the quadratic function $\R \ni x \mapsto x (\varepsilon - x)$ attains its global maximum for $x = \frac{\varepsilon}{2}$, so
\[
	1 + \varepsilon + \frac{2}{1 - \varepsilon} \lambda_{m_1} (\varepsilon - \lambda_{m_1})
		\leq 1 + \varepsilon + \frac{\varepsilon^2}{2 (1 - \varepsilon)}.
\]
Altogether, we conclude that \eqref{eq:support_ineq} is true,
which completes the proof.
\end{proof}

\section{Tightness of Theorem\texorpdfstring{~\ref{thm:main}}{ 1.1}}\label{sec:tightness}

We begin this section with an example showing that the estimate in Theorem~\ref{thm:main} cannot be improved if we restrict ourself
to using asymmetry point simplices.
Verifying all claims in the example is straightforward though somewhat tedious, so we omit the computations.

\begin{example}\label{ex:asymmetry_simplex_tightness}
For $n \geq 2$, let $T = \conv(\{ v^1, \dots, v^{n+1} \})$ be an $n$-simplex centered at the origin.
For $s \in (n-1,n]$,
$\varepsilon := n - s$,
$\rho \in [0,\min\{ \varepsilon, \frac{2 (s-1)}{n+1} \}]$,
and $v := \sum_{i \in [n+1] \setminus [2]} \frac{v^i}{n-1} = - \frac{v^1 + v^2}{n-1}$,
define
\begin{align*}
	x^1 & := \frac{\varepsilon + \rho}{2} v + \frac{2 - \varepsilon - \rho}{2} \bigg( v^1 + \frac{\varepsilon}{1 - \varepsilon} \frac{v^1 - v^2}{2} \bigg),
    \\
	x^2 & := \frac{\varepsilon + \rho}{2} v + \frac{2 - \varepsilon - \rho}{2} \bigg( v^2 + \frac{\varepsilon}{1 - \varepsilon} \frac{v^2 - v^1}{2} \bigg),
    \\
    K & := \conv(\{ x^1, x^2, v^1, \dots, v^{n+1} \}).
\end{align*}
Then $s(K) = s$, $c := - \frac{\varepsilon}{s+1} \frac{v^1 + v^2}{2}$ is the unique Minkowski center of $K$, and $T$ is an asymmetry point simplex of $K$ with
\[
	R(K,T)
    = 1 + \varepsilon + \frac{\varepsilon (\varepsilon - \rho)}{2 (1 - \varepsilon)}.
\]
If $\rho > 0$, then $T$ is the only asymmetry point simplex of $K$.
\end{example}

\begin{remark}
From analysing the proof of Theorem~\ref{thm:main}, it can be shown that any $K \in \CK^n$ with $s(K) \geq n - \varepsilon$ for some $\varepsilon \in (0,1)$ has an asymmetry point simplex $T$ with
\begin{equation}\label{eq:asymmetry_point_simplex_strict}
    R(K,T)
    < 1 + \varepsilon + \frac{\varepsilon^2}{2 (1-\varepsilon)}.
\end{equation}
Indeed, if the initially chosen asymmetry point simplex $T$ does not already satisfy the strict inequality, then having equality in all estimates throughout the proof gives sufficient control over $K$ to show that there must exist another asymmetry point simplex $T'$ that cannot also achieve equality in \eqref{eq:asymmetry_point_simplex_strict}.
In this sense, the uniqueness of the asymmetry point simplex for $\rho > 0$ in the above example is optimal.
Since \eqref{eq:asymmetry_point_simplex_strict} is a minor detail but verifying it is fairly tedious, we omit the proof.
\end{remark}

We now focus on the claimed tightness of the inequality in Theorem~\ref{thm:main} when we do not restrict to asymmetry point simplices.
It is shown in the working manuscript \cite{2026:Grundbacher}
how to compute the Banach--Mazur distance from the triangle to any quadrangle by means of comparing the lengths of the diagonals of the quadrangle to the lengths of the sections of the diagonals induced by the intersection point of the diagonals.
From relating these lengths also to the Minkowski asymmetry of the quadrangle, the following simple formula for trapezoids is obtained in \cite{2026:Grundbacher}.

\begin{proposition} \label{prop:trapezoid}
Let $K \in \CK^2$ be a trapezoid. Then
\[
    d_{BM}(K,\simplex{2})
    = 3 - s(K).
\]
Moreover, if $T \subset K$ is a triangle containing the longer of the two parallel edges of $K$ and one of the vertices of $K$ not incident to that edge,
then $R(K,T) = d_{BM}(K,\simplex{2})$.
\end{proposition}

This already covers the tightness claimed in Theorem~\ref{thm:main} for $n = 2$.
Verifying exact Banach--Mazur distances in higher dimensions is typically much more involved, especially when other parameters like the Minkowski asymmetry must also be controlled.
In our case, we can circumvent such problems by lifting the above planar example to higher dimensions.
Doing so requires a method of embedding convex bodies into higher dimensions under which the behavior of both the Minkowski asymmetry and the Banach--Mazur distance to the simplex is understood.
An appropriate choice of such a method turns out to be building cones over the lower-dimensional convex bodies.

The Minkowski asymmetry of the cones is handled by \cite[Theorem~2]{2002:GuoKaijser}.
The asymmetry simply jumps up by $1$ each time we lift a convex body to the next higher dimension.

\begin{proposition} \label{prop:cone_asymmetry}
Let $L \in \CK^{n-1}$ and $x \in \R^n \setminus \aff(\embedding{L}{0})$.
Then
\[
    s(\coneembedding{L}{0}{x})
    = s(L) + 1.
\]
\end{proposition}

The Banach--Mazur distance from the cones to the simplex is covered by the following special case of Theorem~$1.4$ from the preprint \cite{2026:GrundbacherKobos}.
It states that the distance does not change provided that the distance to the simplex is not larger than $2$ for the initial convex body.
This restriction is of no concern to us since it still covers all trapezoids by Proposition~\ref{prop:trapezoid}.

\begin{proposition} \label{prop:cone_distance}
Let $L \in \CK^{n-1}$ with $d_{BM}(L,\simplex{n-1}) \leq 2$
and $x \in \R^n \setminus \aff(\embedding{L}{0})$.
Then
\[
    d_{BM}(\coneembedding{L}{0}{x}, \simplex{n})
    = d_{BM}(L, \simplex{n-1}).
\]
Moreover, if $T \subset L$ is an $(n-1)$-simplex with $R(L,T) = d_{BM}(L,\simplex{n-1})$, then $R(\coneembedding{L}{0}{x}, \coneembedding{T}{0}{x}) = d_{BM}(\coneembedding{L}{0}{x},\simplex{n})$.
\end{proposition}

Putting all ingredients together provides
a direct proof of the tightness of Theorem~\ref{thm:main} on $\varepsilon$ for all $n$.

\begin{proof}[Proof of the tightness in Theorem~\ref{thm:main}]
We proceed by induction on $n$.
For $n = 2$, we can choose trapezoids according to Proposition~\ref{prop:trapezoid}.
The existence of trapezoids for all Minkowski asymmetry values in $[1,2]$ follows from the continuity of the Minkowski asymmetry on the space of convex bodies (with non-empty interior)
and continuously deforming a square to a triangle while keeping all intermediate bodies as trapezoids.

For the induction step from dimension $n-1 \geq 2$ to $n$,
we take any convex body $L \in \CK^{n-1}$
with the desired properties for dimension $n-1$ and some given $\varepsilon \in [0,1]$, such that
\[
    s(L) = (n-1) - \varepsilon
        \quad \text{and} \quad
    d_{BM}(L,\simplex{n-1}) = 1 + \varepsilon \leq 2.
\]
Propositions~\ref{prop:cone_asymmetry}~and~\ref{prop:cone_distance} show for any $x \in \R^n \setminus \aff(\embedding{L}{0})$ that
$K := \coneembedding{L}{0}{x}$ satisfies
\[
    s(K) = s(L) + 1 = n - \varepsilon
        \quad \text{and} \quad
    d_{BM}(K,\simplex{n}) = d_{BM}(L,\simplex{n-1}) = 1 + \varepsilon.
\]
This completes the induction and thus the proof.
\end{proof}

\section{Stability for the Distance to the Ball}\label{sec:StabBall}

As outlined in the introduction, we obtain Theorem~\ref{thm:AsymmetryDistance} from verifying the following more specific result about volume-minimal circumscribed ellipsoids,
the so-called \cemph{Loewner ellipsoids}.
We say that a convex body $K \in \CK^n$ is in \cemph{Loewner position}
if $\B_2^n$ is the unique ellipsoid of minimal volume containing $K$.
Our goal is to show that $K$ contains a sufficiently large copy of its Loewner ellipsoid depending on the Minkowski asymmetry of $K$.
The main novelty
of our proof is that, in contrast to the usual approach from the literature (cf.~\cite[Theorem~$9$]{2008:BelloniFreund}),
we allow varying centers for the inner ellipsoids instead of restricting to concentricity.
In the context of the Minkowski asymmetry, it is natural that the appropriate choice of the center should combine the center of the Loewner ellipsoid with the Minkowski center in case they deviate from each other.
The precise choice stated below guarantees particularly large inner ellipsoids as can be seen from the proof.

\begin{lemma} \label{lem:befrconj}
Let $K \in \CK^n$ be in Loewner position and let $c \in \R^n$ and $s \geq 1$ be chosen such that $K \subset s \cdot (c - K) + c$.
Then
\[
	\frac{1}{\sqrt{n s}} \B_2^n + \frac{s+1}{2s} c
	\subset K.
\]
\end{lemma}

Before we proceed with the proof, let us point out some special cases of this lemma beside the application for the Minkowski asymmetry.
For $c = 0$ and $s = 1$ (when $K$ is symmetric) or $s = n$ (for general $K$),
the lemma reestablishes the known sharp results about the approximation of convex bodies by their Loewner ellipsoids.
Keeping $c = 0$ but considering the smallest possible choice of $s \in [1,n]$ for a given convex body $K$ smoothly interpolates between these cases and leads to precisely \cite[Theorem~$9$]{2008:BelloniFreund}.
In this case, the minimal $s \in [1,n]$ is also called the \emph{Loewner asymmetry} of $K$.
Further, the Minkowski--Radon theorem shows that we can choose $s = n$ when $c$ is the centroid of $K$.
Therefore, Lemma~\ref{lem:befrconj} implies that a convex body in Loewner position contains Euclidean balls of radius $\frac{1}{n}$ centered at any point between $0$ and $\frac{n+1}{2n}$ times its centroid.

To verify Lemma~\ref{lem:befrconj}, we use and modify the methods employed by Belloni and Freund in the proof of \cite[Theorem~9]{2008:BelloniFreund}.
The first ingredient is, naturally, John's theorem \cite{1948:John} characterizing Loewner ellipsoids.
The second one is \cite[Proposition~4]{2008:BelloniFreund}. We cite them as the two propositions below.

\begin{proposition}\label{prop:lownertheorem}
Let $K \in \CK^n$ be such that $K \subset \B_2^n$.
Then the following are equivalent:
\begin{enumerate}[(i)]
\item $K$ is in Loewner position.
\item There exist some $u^1, \dots, u^m \in K \cap \bd(\B_2^n)$ and $\lambda_1, \dots, \lambda_m > 0$ with
\[
    \sum_{i \in [m]} \lambda_i = n,
        \quad
    \sum_{i \in [m]} \lambda_i u^i = 0,
        \quad \text{and} \quad
    \sum_{i \in [m]} \lambda_i (x^T u^i) u^i = x \text{ for all } x \in \R^n.
\]
\end{enumerate}
\end{proposition}

In general, we say that points $u^1,\dots,u^m \in \bd(\B_2^n)$ together with weights $\lambda_1, \dots, \lambda_m > 0$ form a John decomposition if the conditions in $(ii)$ hold.

\begin{proposition} \label{prop:befr}
Let $\alpha_1, \dots, \alpha_m \in \R$ and let $\alpha_{\max}$, $\alpha_{\min}$ be the largest and smallest among them.
For $p_1, \dots, p_m \geq 0$ with $\sum_{i \in [m]} p_i = 1$,
define $\mu := \sum_{i \in [m]} p_i \alpha_i$ and $\sigma^2 := \sum_{i \in [m]} p_i (\alpha_i - \mu)^2$.
Then
\[
	\sigma^2
	\leq (\alpha_{\max} - \mu) (\mu - \alpha_{\min}).
\]
\end{proposition}

We are now ready to prove Lemma~\ref{lem:befrconj}.

\begin{proof}[Proof of Lemma~\ref{lem:befrconj}]
The assumption $s \geq 1$ directly implies that $\rho := \frac{s+1}{2s} \in [0,1]$.
From considering support functions, the claimed set containment follows once we prove for any $a \in \bd(\B_2^n)$ that
\[
	h_{K - \rho c}(a)
	\geq \frac{1}{\sqrt{n s}}.
\]
Since $K$ is in Loewner position, Proposition~\ref{prop:lownertheorem} yields some contact points $u^1, \dots, u^m \in K \cap \bd(\B_2^n)$ and weights $\lambda_1, \dots, \lambda_m > 0$ that form a John decomposition.
Our goal is to take advantage of Proposition~\ref{prop:befr}.
We define for $i \in [m]$ the scalar $\alpha_i := a^T (u^i - \rho c)$ and the weigth $p_i := \frac{\lambda_i}{n}$,
such that $\sum_{i \in [m]} p_i = 1$.
As in Proposition~\ref{prop:befr}, we define
\begin{align*}
	\mu
	& := \sum_{i \in [m]} p_i \alpha_i
	= \sum_{i \in [m]} \frac{\lambda_i}{n} a^T (u^i - \rho c)
	= \frac{1}{n} a^T \Bigg( \sum_{i \in [m]} \lambda_i u^i \Bigg) - \rho a^T c \sum_{i \in [m]} \frac{\lambda_i}{n}
	= - \rho a^T c
\intertext{
and
}
	\sigma^2
	& := \sum_{i \in [m]} p_i (\alpha_i - \mu)^2
	= \sum_{i \in [m]} \frac{\lambda_i}{n} \big( a^T (u^i - \rho c) + \rho a^T c \big)^2
	= \sum_{i \in [m]} \frac{\lambda_i}{n} (a^T u^i)^2
	\\
	& = \frac{1}{n} a^T \Bigg( \sum_{i \in [m]} \lambda_i (a^T u^i) u^i \Bigg)
	= \frac{a^T a}{n}
	= \frac{1}{n}.
\end{align*}
Next, we observe that
\begin{equation} \label{eq:max_ai}
	\mu
	\leq \max_{i \in [m]} \alpha_i
	= \max_{i \in [m]} a^T (u^i - \rho c)
	\leq h_{K - \rho c}(a).
\end{equation}
Furthermore, we have for any $i \in [m]$ that $u^i \in K \subset s \cdot (c - K) + c$ and consequently
\[
	\rho c - u^i
	\in \rho c - ( s \cdot (c - K) + c )
	= s \cdot (K - \rho c) - (s - s \rho + 1 - \rho) c
	= s \cdot (K - \rho c) - (s+1) (1-\rho) c.
\]
It follows that
\begin{equation} \label{eq:min_ai}
	- \min_{i \in [m]} \alpha_i
	= \max_{i \in [m]} a^T (\rho c - u^i)
	\leq h_{s \cdot (K - \rho c) - (s+1) (1-\rho) c}(a)
	= s h_{K - \rho c}(a) - (s+1) (1-\rho) a^T c.
\end{equation}
Now, Proposition~\ref{prop:befr} together with \eqref{eq:max_ai} and \eqref{eq:min_ai} shows
\begin{align*}
	\frac{1}{n}
	& = \sigma^2
	\leq \Big( \max_{i \in [m]} \alpha_i - \mu \Big) \Big( \mu - \min_{i \in [m]} \alpha_i \Big)
	\leq (h_{K - \rho c}(a) - \mu) \Big( \mu - \min_{i \in [m]} \alpha_i \Big)
	\\
	& \leq (h_{K - \rho c}(a) - \mu) (\mu + s h_{K - \rho c}(a) - (s+1) (1-\rho) a^T c)
	\\
	& = s h_{K - \rho c}(a)^2 + (\mu - (s+1) (1-\rho) a^T c - s \mu) h_{K - \rho c}(a) - \mu^2 + \mu (s+1) (1-\rho) a^T c
	\\
	& = s h_{K - \rho c}(a)^2 + (-\rho - (s+1) (1-\rho) + s \rho) (a^T c) h_{K - \rho c}(a) - \rho^2 (a^T c)^2 - (s+1) (1-\rho) \rho (a^T c)^2.
\end{align*}
Additionally, we observe that $\rho = \frac{s+1}{2s} \in [0,1]$ gives $(1-\rho) \rho \geq 0$ and
\[
	-\rho - (s+1) (1-\rho) + s \rho
	= -\rho - s + s \rho - 1 + \rho + s \rho
	=  2 s \rho - (s + 1)
	= 0.
\]
The above may therefore be simplified to
\begin{equation}\label{eq:ball_last_estimate}
	\frac{1}{n}
	\leq s h_{K - \rho c}(a)^2 - \rho^2 (a^T c)^2 - (s+1) (1-\rho) \rho (a^T c)^2
	\leq s h_{K - \rho c}(a)^2,
\end{equation}
which yields $\vert h_{K - \rho c}(a) \vert \geq 1/\sqrt{n s}$.
Thus, the proof of the claimed set containment is complete once we show $h_{K - \rho c}(a) \geq 0$.
To this end, we observe that $0 \in K$ since $K$ is in Loewner position.
Moreover, if $c \in K$ were wrong, then there would exist some hyperplane $H$ strongly separating $K$ and $c$.
In this case, $s \cdot (c - K) + c$ would clearly lie in the same halfspace bounded by $H$ as $c$.
Thus, $H$ would also strongly separate $K$ and $s \cdot (c - K) + c$, contradicting $K \subset s \cdot (c - K) + c$.
We thus conclude $c \in K$.
Finally, $\rho \in [0,1]$ gives $\rho c \in [0,c] \subset K$,
that is, $0 \in K - \rho c$.
It follows that $h_{K - \rho c}(a) \geq a^T 0 = 0$,
so the claimed set containment follows as discussed above.
\end{proof}

\begin{remark}\label{rem:ellipsoid_approx}
With Lemma~\ref{lem:befrconj}, we can determine for a convex body $K \in \CK^n$ how well its John ellipsoid $\CE_J(K)$ and its Loewner ellipsoid $\CE_L(K)$ approximate the actual Banach--Mazur distance between $K$ and the Euclidean ball.
The John ellipsoid theorem already gives
\[
    d_J(K,\B_2^n)
    := \inf \{ \rho \geq 0 : \CE_J(K) \subset K \subset v + \rho \CE_J(K), v \in \R^n \}
    \leq n
    \leq n \cdot d_{BM}(K,\B_2^n).
\]
Remarkably, the estimate from left to right is already of the optimal order in $n$,
as the factor $n$ cannot be reduced below $\frac{n+1}{4}$ for example for the convex body $K := \conv( \B_2^n \cup \{n u\})$ when $u \in \bd(\B_2^n)$.
It is easy to verify that $\CE_J(K) = \B_2^n$, $d_J(K,\B_2^n) \geq \frac{n+1}{2}$, and $d_{BM}(K,\B_2^n) \leq 2$.

In contrast, the approximation by Loewner ellipsoids provides a better worst case order in $n$.
Lemma~\ref{lem:befrconj} together with the well-known inequality $s(K) \leq d_{BM}(K,\B_2^n)$ (see, e.g., \cite[Proposition~$3.1$]{2017:BrandenbergGonzalezMerino}) 
shows that
\begin{align*}
    d_L(K,\B_2^n)
    & := \inf \{ \rho \geq 0 : u + \rho^{-1} \CE_L(K) \subset K \subset \CE_L(K), u \in \R^n \}
    \\
    & \leq \sqrt{n s(K)}
    \leq \sqrt{n} \cdot s(K)
    \leq \sqrt{n} \cdot d_{BM}(K,\B_2^n).
\end{align*}
In particular, when we take the better of the approximations by John and Loewner ellipsoids, we always have
\[
    \min \{ d_J(K,\B_2^n), d_L(K,\B_2^n) \}
    \leq \sqrt{n} \cdot d_{BM}(K,\B_2^n).
\]
Even this inequality is again of the optimal order in $n$:
The two quantities in the minimum equal $\sqrt{n-1}$ for
$K := \conv(\{ \pm \sqrt{n-1} \cdot u \} \cup ( \B_2^n \cap H_{(\pm v, 1/\sqrt{n-1})}^\leq ) )$ when $u,v \in \bd(\B_2^n)$ are orthogonal,
whereas this body again satisfies $d_{BM}(K,\B_2^n) \leq 2$.

Let us point out that the tightness result for John ellipsoids shows that the analog of Lemma~\ref{lem:befrconj} for John ellipsoids is \emph{not} true.
Otherwise, we would get order $\sqrt{n}$ like for the Loewner ellipsoid.

Further, notice that the example of tightness for John ellipsoids shows that Lemma~\ref{lem:befrconj} would be wrong if we insisted on using $0$ as center for the inballs.
Indeed, by polarity, it shows for any vector $u \in \bd(\B_2^n)$ that the convex body $K := \B_2^n \cap H_{(u,1/n)}^\leq$ is in Loewner position.
The largest ball contained in $K$ with center at $0$ has radius $\frac{1}{n}$.
However, $s(K) \leq d_{BM}(K,\B_2^n) \leq 2$ since $K$ is a body of revolution, so that $\sqrt{n s(K)} < n$ for $n \geq 3$.
\end{remark}

We end this section with the immediate consequences on the Banach--Mazur distance to $\B_2^n$, namely Theorems~\ref{thm:DistanceBall}~and~\ref{thm:AsymmetryDistance}.

\begin{proof}[Proof of Theorem~\ref{thm:AsymmetryDistance}]
We may assume that $K$ is in Loewner position.
Choosing $c \in \R^n$ to be any Minkowski center of $K$ and $s := s(K)$,
we have $K \subset s \cdot (c - K) + c$.
In this case, Lemma~\ref{lem:befrconj} shows that $K$ contains a Euclidean ball of radius $1/\sqrt{n s(K)}$.
Since $K$ is in Loewner position, we also have $K \subset \B_2^n$.
Altogether, the claimed inequality follows.

For the claimed sharpness, let $s \in [1,n]$ be an integer divisor of $n$, let $k := \frac{n}{s}$, and choose $K$ to be the convex hull of $k$ many $s$-simplices $T^1, \dots, T^k$ that are contained in $k$ respective, pairwise orthogonal subspaces and all have their Minkowski centers at $0$.
Since all $T^i$ are Minkowski centered with $s(T^i) = s$,
the reflected set $-sK$ clearly contains $T^i$ for all $i \in [k]$ and therefore also $K$.
It follows that $s(K) \leq s$.
For the reverse inequality, let $c \in \R^n$ such that $K \subset c - s(K) K$.
With $P$ the orthogonal projection onto the linear hull of $T^1$, we have $P(T^i) = \{0\}$ for all $i \geq 2$, so that
\[
    T^1
    = P(K)
    \subset P(c - s(K) K)
    = P(c) - s(K) T^1.
\]
Therefore, $s(K) \geq s(T^1) = s$, which verifies $s(K) = s$.
Lastly, since $d_{BM}(T^i,\B_2^s) = s$ for all $i \in [k]$, where this distance is realized by a pair of $0$-symmetric ellipsoids for each $i \in [k]$, it follows from \cite[Remark~$2.5$]{2026:GrundbacherKobos} that $d_{BM}(K,\B_2^n) = \sqrt{ k s^2 } = \sqrt{n s(K)}$, completing the proof.
\end{proof}

\begin{remark}\label{rem:Litvak}
For general $s \in [n]$, constructing $K \in \CK^n$ from simplices like above but using $\lfloor \frac{n}{s} \rfloor$ many $s$-simplices and one $(n - s \lfloor \frac{n}{s} \rfloor)$-simplex leads with essentially the same proof to $s(K) = s$ and
\[
    d_{BM}(K,\B_2^n)
    = \sqrt{ns - \bigg( n - s \bigg\lfloor \frac{n}{s} \bigg\rfloor \bigg) \bigg( s \bigg\lceil \frac{n}{s} \bigg\rceil - n \bigg)}.
\]
We remark that while this does not fully reach $\sqrt{n s}$ when $s$ does not divide $n$,
as a consequence of the inequality between the geometric mean and the arithmetic mean
together with a case distinction whether $s \leq \frac{n}{2}$,
we see that this distance is still at least
\[
    \sqrt{2 \big( \sqrt{2}-1 \big)} \cdot \sqrt{n s}
    > 0.91 \cdot \sqrt{n s}.
\]
\end{remark}

\begin{proof}[Proof of Theorem~\ref{thm:DistanceBall}]
Theorem~\ref{thm:AsymmetryDistance} shows that $n - \varepsilon \leq \sqrt{n s(K)}$.
Squaring both sides and rearranging gives
\[
	s(K)
	\geq n - 2 \varepsilon + \frac{\varepsilon^2}{n}
	> n - 2 \varepsilon
	> n - 1.
\]
Theorem~\ref{thm:main} therefore yields
\[
	d_{BM}(K,\simplex{n})
	< 1 + 2 \varepsilon + \frac{4 \varepsilon^2}{2 (1 - 2 \varepsilon)},
\]
which proves the claimed inequality.

When we restrict to $\varepsilon \leq \gamma$ for some $\gamma \in (0,\frac{1}{2})$,
then
\[
    d_{BM}(K,\simplex{n})
    < 1 + \bigg(2 + \frac{2 \gamma}{1 - 2 \gamma} \bigg) \varepsilon
\]
is a valid upper bound of linear order in $\varepsilon$.
To see that this order is optimal,
let $K \in \CK^n$ with $d_{BM}(K,\B_2^n) = n - \varepsilon$.
Then, by the multiplicative triangle inequality for the Banach--Mazur distance,
\[
	n
	= d_{BM}(\B_2^n,\simplex{n})
	\leq d_{BM}(K,\B_2^n) \cdot d_{BM}(K,\simplex{n})
	= (n - \varepsilon) d_{BM}(K,\simplex{n})
\]
and thus
\[
	d_{BM}(K,\simplex{n})
	\geq \frac{n}{n-\varepsilon}
	= 1 + \frac{\varepsilon}{n - \varepsilon}
	\geq 1 + \frac{\varepsilon}{n}. \qedhere
\]
\end{proof}

\section{Stability for the \texorpdfstring{$m$}{m}-th Order \texorpdfstring{$L_p$}{Lp}-Rogers--Shephard Inequality}\label{sec:RS}

Our goal in this section is to prove Theorems~\ref{thm:Lp_RSS}~and~\ref{thm:Lp_RSS_stability}.
In preparation, we introduce some notation and collect results from \cite{2025:Kotrbaty,2026:KotrbatyMouamine}.
Throughout, the integer $m \geq 1$ and the reals $p,q \in [1,\infty]$ such that $\frac{1}{p} + \frac{1}{q} = 1$ are understood as fixed parameters that are omitted from newly introduced notation for brevity.
We understand $\frac{x}{\infty} = 0$ for $x \in \R$.
We denote by $V(K[k],L[n-k])$ the \cemph{mixed volume} of $k \in \{0\} \cup [n]$ copies of $K \subset \R^n$ and $n-k$ copies of $L \subset \R^n$ when the sets are non-empty, compact, and convex (cf.~\cite[Chapter~$5$]{2014:Schneider}).
Using the representation of $p$-sums from \cite{2012:LutwakYangZhang}, we have
\[
    \diff{p}{m}{K}
    = \bigcup_{t \in [0,1]} \lyz{p}{m}{K}{t},
        \quad \text{where} \quad
    \lyz{p}{m}{K}{t}
    := (1-t)^{\frac{1}{q}} \diag{m}{K} - t^{\frac{1}{q}} \cartesian{K}{m}
    \text{ for all } t \in [0,1].
\]
For $z \in \diff{p}{m}{K}$, we define the set
\[
    \interval{p}{m}{K}{z}
    := \{ t \in [0,1] : z \in \lyz{p}{m}{K}{t} \}.
\]
It can be shown to be an interval, though we do not need this fact later.

Whenever $T \in \CK^n$ is a simplex with $0 \in T$ in this section,
we let its vertices be given by $v^1, \dots, v^{n+1} \in \R^n$.
Moreover, the facet outer normals of $T$ are always given by $a^1, \dots, a^{n+1} \in \R^n$ such that for all $i,j \in [n+1]$ with $i \neq j$,
we have $0 \leq (a^i)^T v^j = (a^i)^T v^i + 1 =: \beta_i$.
In particular, $T = \bigcap_{i \in [n+1]} H_{(a^i,\beta_i)}^\leq$.
This normalization of the $a^i$ results in
\begin{equation} \label{eq:simplex_normalization}
    \sum_{i \in [n+1]} a^i = 0
        \quad \text{and} \quad
    \sum_{i \in [n+1]} \beta_i = 1.
\end{equation}
Indeed, if $c_T$ is the (Minkowski) center of $T$, then $T - c_T$ is centered at the origin and
\begin{equation}\label{eq:simplex_normalization_centered}
    1
    = h_T(a^i) + h_T(-a^i)
    = h_{T-c_T}(a^i) + h_{T-c_T}(-a^i)
    = (n+1) h_{T-c_T}(a^i)
    = (n+1) (\beta_i - (a^i)^T c_T).
\end{equation}
The proof of Lemma~\ref{lem:simplex_circumradius} shows from $h_{T-c_T}(a^i) = \frac{1}{n+1}$ for all $i \in [n+1]$ that $\sum_{i \in [n+1]} (n+1) a^i = 0$
and therefore the first part of \eqref{eq:simplex_normalization}.
With \eqref{eq:simplex_normalization_centered}, we further obtain
$1 = \sum_{i \in [n+1]} \beta_i - (a^i)^T c_T = \sum_{i \in [n+1]} \beta_i$,
so the second part of \eqref{eq:simplex_normalization} is also true.

The following proposition for $k \leq n$ is the inequality part of \cite[Conjecture~$2$]{2025:Kotrbaty}, a generalization of the Godbersen conjecture.
It is stated in \cite{2026:KotrbatyMouamine} that the method used there to verify the inequality part of the usual Godbersen conjecture directly generalizes for the even more general \cite[Conjecture~$1.3$]{2025:Kotrbaty}, which we verified independently.
We give the resulting statement here with some further details.
The inequality for $k > n$, where both sides equal $0$, follows from \cite[Theorem~$5.1.8$]{2014:Schneider}.
The equality case for simplices is obtained from \cite[Proposition~$3.2$]{2025:Kotrbaty} together with the equality case in Schneider's inequality \eqref{eq:Schneider_RS}.

\begin{proposition}\label{prop:Godbersen}
Let $K \in \CK^n$ and $k \in \{0\} \cup [nm]$.
Then
\[
    V(-\diag{m}{K}[k],\cartesian{K}{m}[nm-k])
    \leq \binom{n}{k} \vol_n(K)^m,
\]
with equality when $K$ is a simplex.
\end{proposition}

For $k = 1$, the inclusion $-\diag{m} K + \diag{m} c \subset s(K) \diag{m} K \subset s(K) \cartesian{K}{m}$ for $c \in \R^n$ with $-K + c \subset s(K) K$ leads to the stronger inequality
\begin{equation}\label{eq:asymmetry_Godbersen}
    V(-\diag{m} K[1], \cartesian{K}{m}[nm-1])
    \leq s(K) V(\cartesian{K}{m}[nm])
    = s(K) \vol_{nm}(\cartesian{K}{m})
    = s(K) \vol_n(K)^m.
\end{equation}

As outlined in the introduction, the method for proving Theorem~\ref{thm:Lp_RSS} is essentially the one from \cite{2026:KotrbatyMouamine} for \eqref{eq:Lp_RS},
up to minor simplifications.
The approach presented there directly generalizes to the higher-order setting if Proposition~\ref{prop:Godbersen} is substituted for the verified inequality from the usual Godbersen conjecture.
The proof proceeds by finding a lower bound on the \cemph{length} $|\interval{p}{m}{K}{z}|$ for all $z \in \diff{p}{m}{K}$, and then integrating over both $|\interval{p}{m}{K}{z}|$ and this lower bound to obtain an upper bound on
$\vol_{nm}(\diff{p}{m}{K})$ in terms of the $\vol_{nm}(\lyz{p}{m}{K}{t})$.
Afterwards, applying the multilinearity of the mixed volume and Proposition~\ref{prop:Godbersen} to bound the $\vol_{nm}(\lyz{p}{m}{K}{t})$ yields the inequality.
The proof of Theorem~\ref{thm:Lp_RSS_stability} then uses \eqref{eq:asymmetry_Godbersen} to control the distance to the simplex, and a stability improvement of the bound on $|\interval{p}{m}{K}{z}|$ for simplices to control the positioning of $K$.

Our first goal is therefore to obtain a lower bound for $|\interval{p}{m}{K}{z}|$ analogous to the one from \cite{2026:KotrbatyMouamine}.
The method for general bodies $K$ is essentially the same albeit simplified.
For simplices, we also require the corresponding equality case and a stability improvement based on the positioning.
Therefore, we first provide a novel lemma characterizing the membership in $\lyz{p}{m}{T}{t}$ for a simplex $T$.

\begin{lemma}\label{lem:simplex_lyz}
Let $T \in \CK^n$ be a simplex with $0 \in T$,
$z = (z^1, \dots, z^m) \in \cartesian{(\R^n)}{m}$,
and $t \in [0,1]$.
Then
\[
	z \in \lyz{p}{m}{T}{t}
	\iff 0 \leq \sumfunc{p}{m}{T}{z}{t}
	:= \sum_{i \in [n+1]} \min \bigg\{ (1-t)^{\frac{1}{q}} \beta_i, t^{\frac{1}{q}} \beta_i + \min_{j \in [m]} (a^i)^T z^j \bigg\}.
\]
\end{lemma}
\begin{proof}
First, assume that $z \in \lyz{p}{m}{T}{t}$.
Then there exist $x, y^1, \dots, y^m \in T$ with $z^j = (1-t)^{\frac{1}{q}} x - t^{\frac{1}{q}} y^j$ for all $j \in [m]$.
For any $i \in [n+1]$ and $j \in [m]$, we have
\[
	(1-t)^{\frac{1}{q}} (a^i)^T x \leq (1-t)^{\frac{1}{q}} \beta_i
		\quad \text{and} \quad
	(1-t)^{\frac{1}{q}} (a^i)^T x = t^{\frac{1}{q}} (a^i)^T y^j + (a^i)^T z^j \leq t^{\frac{1}{q}} \beta_i + (a^i)^T z^j.
\]
In particular, if we always choose $j = j(i) \in [m]$ to minimize $(a^i)^T z^j$,
then $\sum_{i \in [n+1]} a^i = 0$ leads to
\[
	0
	= \sum_{i \in [n+1]} (1-t)^{\frac{1}{q}} (a^i)^T x
	\leq \sum_{i \in [n+1]} \min \bigg\{ (1-t)^{\frac{1}{q}} \beta_i, t^{\frac{1}{q}} \beta_i + \min_{j \in [m]} (a^i)^T z^j \bigg\}
	= \sumfunc{p}{m}{T}{z}{t}.
\]

Now, assume that $\sumfunc{p}{m}{T}{z}{t} \geq 0$.
Choose $x \in \R^n$ as the unique solution of the $n$ linear equations
\[
	(a^i)^T x
	= \min \bigg\{ (1-t)^{\frac{1}{q}} \beta_i, t^{\frac{1}{q}} \beta_i + \min_{j \in [m]} (a^i)^T z^j \bigg\} - \frac{\sumfunc{p}{m}{T}{z}{t}}{n+1}
\]
for $i \in [n]$.
Since $\sum_{i \in [n+1]} a^i = 0$, this equation must also be satisfied for $i = n+1$.
Additionally, we define $y^j = x - z^j$ for all $j \in [m]$.
Then $\sumfunc{p}{m}{T}{z}{t} \geq 0$ implies $(a^i)^T x \leq (1-t)^{\frac{1}{q}} \beta_i$ for all $i \in [n+1]$ and thus $x \in (1-t)^{\frac{1}{q}} T$.
Moreover, for all $i \in [n+1]$ and $j \in [m]$,
\[
	(a^i)^T y^j
	= (a^i)^T x - (a^i)^T z^j
	\leq t^{\frac{1}{q}} \beta_i + \min_{k \in [m]} (a^i)^T z^k - (a^i)^T z^j
	\leq t^{\frac{1}{q}} \beta_i,
\]
which yields $y^j \in t^{\frac{1}{q}} T$.
Consequently,
\[
	z
	= (z^1, \dots, z^m)
	= (x, \dots, x) - (y^1, \dots, y^m)
	\in (1-t)^{\frac{1}{q}} \diag{m} T - t^{\frac{1}{q}} \cartesian{T}{m}
	= \lyz{p}{m}{T}{t}
\]
completes the proof.
\end{proof}

\begin{lemma}\label{lem:general_iz}
Let $K \in \CK^n$ with $0 \in K$, $z \in \diff{p}{m}{K}$,
and suppose that $p > 1$ (and $q < \infty$).
Then
\[
	|\interval{p}{m}{K}{z}|
	\geq 1 - \|z\|_{\diff{p}{m}{K}}^q,
\]
with equality if $K$ is a simplex with a vertex at $0$.
\end{lemma}
\begin{proof}
There exist $x \in \diag{m}{K}$, $y \in \cartesian{K}{m}$, and $t_0 \in [0,1]$ with
\[
	z
	= \|z\|_{\diff{p}{m}{K}} \Big( (1-t_0)^{\frac{1}{q}} x - t_0^{\frac{1}{q}} y \Big).
\]
If $t \in [0,1]$ satisfies $\|z\|_{\diff{p}{m}{K}}^q t_0 \leq t \leq 1 - \|z\|_{\diff{p}{m}{K}}^q (1-t_0)$, then
\[
	\|z\|_{\diff{p}{m}{K}} (1-t_0)^{\frac{1}{q}} \leq (1-t)^{\frac{1}{q}}
		\quad \text{and} \quad
	\|z\|_{\diff{p}{m}{K}} t_0^{\frac{1}{q}} \leq t^{\frac{1}{q}}.
\]
Together with $0 \in \diag{m}{K} \cap \cartesian{K}{m}$, it follows that
\[
	z
	\in \|z\|_{\diff{p}{m}{K}} (1-t_0)^{\frac{1}{q}} \diag{m}{K} - \|z\|_{\diff{p}{m}{K}} t_0^{\frac{1}{q}} \cartesian{K}{m}
	\subset (1-t)^{\frac{1}{q}} \diag{m}{K} - t^{\frac{1}{q}} \cartesian{K}{m}
	= \lyz{p}{m}{K}{t}.
\]
This means $\big [\|z\|_{\diff{p}{m}{K}}^q t_0, 1 - \|z\|_{\diff{p}{m}{K}}^q (1-t_0) \big] \subset \interval{p}{m}{K}{z}$, which yields the claimed inequality.

For the equality case, we may assume that $K = T$ is a simplex and $v^{n+1} = 0$.
Then $\beta_i = 0$ for all $i \in [n]$, and $\beta_{n+1} = 1$ from $\sum_{i \in [n+1]} \beta_i = 1$, so $\sum_{i\in[n+1]} a^i = 0$ shows
\[
	\sumfunc{p}{m}{T}{z_0}{t_0}
	= \min \bigg\{ (1-t_0)^{\frac{1}{q}}, t_0^{\frac{1}{q}} - \max_{j \in [m]} \sum_{i \in [n]} (a^i)^T z_0^j \bigg\} + \sum_{i \in [n]} \min \bigg\{ 0, \min_{j \in [m]} (a^i)^T z_0^j \bigg\}
\]
for any $z_0 = (z_0^1, \dots, z_0^m) \in \cartesian{(\R^n)}{m}$ and $t_0 \in [0,1]$.
By Lemma~\ref{lem:simplex_lyz}, we have $z_0 \in \diff{p}{m}{T}$ if and only if there exists $t_0 \in [0,1]$ with $\sumfunc{p}{m}{T}{z_0}{t_0} \geq 0$, that is,
\[
	\Bigg( \max_{j \in [m]} \sum_{i \in [n]} (a^i)^T z_0^j - \sum_{i \in [n]} \min \bigg\{ 0, \min_{j \in [m]} (a^i)^T z_0^j \bigg\} \Bigg)^q
	\leq t_0
	\leq 1 - \Bigg( - \sum_{i \in [n]} \min \bigg\{ 0, \min_{j \in [m]} (a^i)^T z^j \bigg\} \Bigg)^q.
\]
Writing $f_1(z_0)$ and $1 - f_2(z_0)$ for the left- and right-hand terms, respectively,
this means that $z_0 \in \diff{p}{m}{T}$ is equivalent to $f_1(z_0) + f_2(z_0) \leq 1$.
Since $f_1$ and $f_2$ are positively homogeneous of degree $q$,
it follows that $f_1(z_0) + f_2(z_0) = \|z_0\|_{\diff{p}{m}{T}}^q$.
Indeed, we have for any $\rho \geq 0$ that $z_0 \in \rho \diff{p}{m}{T}$ if and only if $f_1(z_0) + f_2(z_0) \leq \rho^q$. The smallest $\rho \geq 0$ satisfying these conditions is $\|z\|_{\diff{p}{m}{T}}$ by the former, and $(f_1(z_0) + f_2(z_0))^{\frac{1}{q}}$ by the latter.
Lastly, we also see from the above for $z \in \diff{p}{m}{T}$ that
\[
	|\interval{p}{m}{T}{z}|
	= (1 - f_2(z)) - f_1(z)
	= 1 - \|z\|_{\diff{p}{m}{T}}^q,
\]
completing the proof of the lemma.
\end{proof}

For the stability improvement of the above lemma,
given a simplex $T \in \CK^n$ with $0 \in T$,
we define for any $\mu \geq 0$ the set
\[
	\simplexscale{T}{\mu}
	:= \mu (T - v^{n+1})
    = \bigcap_{i \in [n]} H_{(a^i,0)}^\leq \cap H_{(a^{n+1},\mu)}^\leq.
\]
Moreover, for $z = (z^1, \dots, z^m) \in \cartesian{(\R^n)}{m}$ we write
\[
	\simplexscaling{T}{z}
	:= \max_{j \in [m]} (a^{n+1})^T z^j.
\]

\begin{lemma}\label{lem:simplex_iz}
Let $T \in \CK^n$ be a simplex with $0 \in T$,
$z \in - \cartesian{\simplexscale{T}{\min \{ \beta_{n+1}, 2^{-\frac{1}{q}} \}}}{m}$,
and suppose that $p > 1$ (and $q < \infty$).
Then $z \in \diff{p}{m}{T}$ and
\[
    |\interval{p}{m}{T}{z}|
	\geq 1 - \|z\|_{\diff{p}{m}{T}}^q + \simplexscaling{T}{-z}^q \min \bigg\{ \frac{1-\beta_{n+1}}{\beta_{n+1}}, \frac{\beta_{n+1}}{1-\beta_{n+1}} \bigg\}.
\]
\end{lemma}
\begin{proof}
Clearly, any $x \in \simplexscale{T}{\min \{ \beta_{n+1}, 2^{-\frac{1}{q}} \}}$ satisfies $(a^i)^T x \leq 0 \leq \beta_i$ for all $i \in [n]$.
Moreover, $(a^{n+1})^T x \leq \beta_{n+1}$,
so that $x \in T$.
Thus, $-\cartesian{\simplexscale{T}{\min \{ \beta_{n+1}, 2^{-\frac{1}{q}} \}}}{m} \subset -\cartesian{T}{m} = \lyz{p}{m}{T}{1} \subset \diff{p}{m}{T}$.

Next, we derive a lower bound on $|\interval{p}{m}{T}{z}|$ from  Lemma~\ref{lem:simplex_lyz}.
Writing $z = (z^1, \dots, z^m) \in \cartesian{(\R^n)}{m}$, we use $(a^i)^T z^j \geq 0$ for $i \in [n]$ and $j \in [m]$ by $z \in -\cartesian{\simplexscale{T}{\min\{ \beta_{n+1}, 2^{-\frac{1}{q}} \}}}{m}$ together with $\sum_{i \in [n+1]} \beta_i = 1$ to estimate for $t \in [0,1]$ that
\begin{align*}
	\sumfunc{p}{m}{T}{z}{t}
	& \geq  \min \bigg\{ (1-t)^{\frac{1}{q}} \beta_{n+1}, t^{\frac{1}{q}} \beta_{n+1} + \min_{j \in [m]} (a^{n+1})^T z^j \bigg\}
		+ \sum_{i \in [n]} \min \Big\{ (1-t)^{\frac{1}{q}}, t^{\frac{1}{q}} \Big\} \beta_i
	\\
	& = \min \bigg\{ (1-t)^{\frac{1}{q}} \beta_{n+1}, t^{\frac{1}{q}} \beta_{n+1} - \simplexscaling{T}{-z} \bigg\}
		+ \min \Big\{ (1-t)^{\frac{1}{q}}, t^{\frac{1}{q}} \Big\} (1 - \beta_{n+1}).
\end{align*}
If the left-hand minimum in the term in the last line is given by $(1-t)^{\frac{1}{q}} \beta_{n+1}$,
then the term in the last line is clearly non-negative, so that $t \in \interval{p}{m}{T}{z}$.
Otherwise, the term in the last line is non-negative if and only if
\begin{equation}\label{eq:muz_cond}
	t^{\frac{1}{q}} \beta_{n+1} + \min \Big\{ (1-t)^{\frac{1}{q}}, t^{\frac{1}{q}} \Big\} (1 - \beta_{n+1})
	\geq \simplexscaling{T}{-z}.
\end{equation}
Note that $t^{\frac{1}{q}} \beta_{n+1} + (1-t)^{\frac{1}{q}} (1-\beta_{n+1})$ is a concave function in $t \in [0,1]$
that attains its minimum on the interval $[\frac{1}{2},1]$ at one of the endpoints.
The respective function values are $2^{-\frac{1}{q}}$ and $\beta_{n+1}$,
both of which are at least $\min\{ \beta_{n+1}, 2^{-\frac{1}{q}} \} \geq \simplexscaling{T}{-z}$.
Hence, \eqref{eq:muz_cond} is satisfied for all $t \in [\frac{1}{2},1]$.
If $t \leq \frac{1}{2}$, then \eqref{eq:muz_cond} is equivalent to $t \geq \simplexscaling{T}{-z}^q$.
In summary, $[\simplexscaling{T}{-z}^q,1] \subset \interval{p}{m}{T}{z}$, such that
\begin{equation}\label{eq:iz_muz_bound}
	|\interval{p}{m}{T}{z}|
	\geq 1 - \simplexscaling{T}{-z}^q.
\end{equation}

Now, choose $k \in [m]$ with $\simplexscaling{T}{-z} = (a^{n+1})^T (-z^k)$ and let $\hat a \in \cartesian{(\R^n)}{m}$ be the vector
whose $k$-th $n$-dimensional vector component is $-a^{n+1}$ and all others are $0 \in \R^n$.
Then
\[
	h_{\diff{p}{m}{T}}(\hat a)
	= \big( h_{\diag{m}{T}} (\hat a)^p + h_{-\cartesian{T}{m}} (\hat a)^p \big)^{\frac{1}{p}}
	= \big( h_T(-a^{n+1})^p + h_T(a^{n+1})^p \big)^{\frac{1}{p}}
	= \big( (1-\beta_{n+1})^p + \beta_{n+1}^p \big)^{\frac{1}{p}},
\]
where we used that $\beta_{n+1} = h_T(a^{n+1}) = 1 - h_T(-a^{n+1})$ by the normalization of the $a^i$.
Therefore,
\[
	\|z\|_{\diff{p}{m}{T}}
	\geq \frac{\hat a^T z}{h_{\diff{p}{m}{T}} (\hat a)}
	= \frac{(-a^{n+1})^T z^k}{h_{\diff{p}{m}{T}} (\hat a)}
	= \frac{\simplexscaling{T}{-z}}{( (1-\beta_{n+1})^p + \beta_{n+1}^p )^{\frac{1}{p}}},
\]
and with \eqref{eq:iz_muz_bound} consequently
\begin{align*}
	|\interval{p}{m}{T}{z}|
	& \geq 1 - \simplexscaling{T}{-z}^q - \|z\|_{\diff{p}{m}{T}}^q + \frac{\simplexscaling{T}{-z}^q}{( (1-\beta_{n+1})^p + \beta_{n+1}^p )^{\frac{q}{p}}}
    \\
	& = 1 - \|z\|_{\diff{p}{m}{T}}^q + \simplexscaling{T}{-z}^q \Big( ( (1-\beta_{n+1})^p + \beta_{n+1}^p )^{-\frac{q}{p}} - 1 \Big).
\end{align*}
Lastly, an easy computation shows for $\beta \in [0,\frac{1}{2}]$ that
$((1-\beta)^p + \beta^p )^{\frac{q}{p}} = ((1-\beta)^p + \beta^p )^{\frac{1}{p-1}} \leq 1-\beta$.
By symmetry, $((1-\beta)^p + \beta^p )^{\frac{q}{p}} \leq \beta$ for $\beta \in [\frac{1}{2},1]$.
The claimed inequality follows.
\end{proof}

We now have all ingredients necessary to verify Theorem~\ref{thm:Lp_RSS}.
Note that the use of Lemma~\ref{lem:simplex_iz} in the characterization of the equality case can be replaced by a simpler argument based on continuity like in \cite{2026:KotrbatyMouamine}.
Since we anyway require Lemma~\ref{lem:simplex_iz} for the proof of Theorem~\ref{thm:Lp_RSS_stability}, we may also use it here.
Before we proceed, we note from the relation of the beta and gamma functions that
\begin{equation}\label{eq:beta_gamma}
    \bigg( 1 + \frac{nm}{q} \bigg) \int_0^1 (1-t)^\frac{k}{q} t^{\frac{nm-k}{q}} \dd t
    = \frac{\Gamma(1+\frac{k}{q}) \Gamma(1+\frac{nm-k}{q})}{\Gamma(1+\frac{nm}{q})}
    = \binom{nm/q}{k/q}^{-1}.
\end{equation}

\begin{proof}[Proof of Theorem~\ref{thm:Lp_RSS}]
Using Tonelli's theorem, we have
\begin{equation}\label{eq:fiber_integral}
    \int_0^1 \vol_{nm}(\lyz{p}{m}{K}{t}) \dd t
    = \int_0^1 \int_{\cartesian{(R^n)}{m}} \bone_{\{ z \in \lyz{p}{m}{K}{t} \}} \dd z \dd t
    = \int_{\diff{p}{m}{K}} |\interval{p}{m}{K}{z}| \dd z.
\end{equation}
Our goal is to use Lemma~\ref{lem:general_iz} to obtain a lower bound to the right-hand integral.
In preparation, we compute for $p > 1$ (and $q < \infty$) with the layer cake representation for the $q$-norm that
\begin{align*}
    \int_{\diff{p}{m}{K}} \|z\|_{\diff{p}{m}{K}}^q \dd z
    & = q \int_0^1 t^{q-1} \vol_{nm}(\diff{p}{m}{K} \setminus t \diff{p}{m}{K}) \dd t
    = q \vol_{nm}(\diff{p}{m}{K}) \int_0^1 t^{q-1} (1 - t^{nm}) \dd t
    \\
    & = q \vol_{nm}(\diff{p}{m}{K}) \bigg( \frac{1}{q} - \frac{1}{nm+q} \bigg)
    = \frac{nm}{nm+q} \vol_{nm}(\diff{p}{m}{K}).
\end{align*}
Therefore, Lemma~\ref{lem:general_iz} and \eqref{eq:fiber_integral} yield
\begin{equation}\label{eq:fiber_estimate}
    \vol_{nm}(\diff{p}{m}{K})
    = \bigg( 1 + \frac{nm}{q} \bigg) \int_{\diff{p}{m}{K}} 1 - \|z\|_{\diff{p}{m}{K}}^q \dd z
    \leq \bigg( 1 + \frac{nm}{q} \bigg) \int_0^1 \vol_{nm}(\lyz{p}{m}{K}{t}) \dd t.
\end{equation}
The inequality from left to right is also true for $p = 1$ and $q = \infty$ (with equality) since $\lyz{p}{m}{K}{t} = \diff{p}{m}{K}$ for all $t \in [0,1]$ in this case.
Hence, using the multilinearity of the mixed volume, Proposition~\ref{prop:Godbersen}, \eqref{eq:asymmetry_Godbersen}, and \eqref{eq:beta_gamma},
we obtain for general $p$ and $q$ that
\begin{align*}
    \vol_{nm}(\diff{p}{m}{K})
    & \leq \bigg(1 + \frac{nm}{q} \bigg) \int_0^1 \vol_{nm}(\lyz{p}{m}{K}{t}) \dd t
    \\
    & = \bigg(1 + \frac{nm}{q} \bigg)
    \int_0^1 \sum_{k \in \{0\} \cup [nm]} \binom{nm}{k} V \Big(-t^\frac{1}{q} \diag{m} K[k], (1-t)^\frac{1}{q} \cartesian{K}{m}[nm-k] \Big) \dd t
    \\
    & = \sum_{k \in \{0\} \cup [n]} \binom{nm}{k} V(-\diag{m} K[k], \cartesian{K}{m}[nm-k])
    \bigg(1 + \frac{nm}{q} \bigg) \int_0^1 t^\frac{k}{q} (1-t)^\frac{nm-k}{q} \dd t
    \\
    & \leq nm \binom{nm/q}{1/q}^{-1}  (s(K)-n) \vol_n(K)^m + \sum_{k \in \{0\} \cup [n]} \binom{nm}{k} \binom{n}{k} \binom{nm/q}{k/q}^{-1} \vol_n(K)^m
    \\
    & = \bigg( C_{n,m,q} - nm \binom{nm/q}{1/q}^{-1} (n-s(K)) \bigg) \vol_n(K)^m
    \leq C_{n,m,q} \vol_n(K)^m.
\end{align*}

Equality is attained if and only if all three estimates above are satisfied with equality, which means equality in Proposition~\ref{prop:Godbersen}, \eqref{eq:asymmetry_Godbersen}, \eqref{eq:fiber_estimate}, and $s(K) \leq n$.
This is simultaneously the case for all three of Proposition~\ref{prop:Godbersen}, \eqref{eq:asymmetry_Godbersen}, and $s(K) \leq n$ if and only if $K$ is a simplex.
Under the assumption that $K$ is a simplex, Lemmas~\ref{lem:general_iz}~and~\ref{lem:simplex_iz} show that equality is attained in \eqref{eq:fiber_estimate} if and only if $p = 1$ or $K$ has a vertex at $0$.
The characterization of the equality case follows.
\end{proof}

Next, we work toward Theorem~\ref{thm:Lp_RSS_stability}.
The proof of Theorem~\ref{thm:Lp_RSS} already ties the tightness of the inequality there to the Minkowski asymmetry.
Near-extremal bodies must have near-maximal Minkowski asymmetry,
so that Theorem~\ref{thm:main} bounds their Banach--Mazur distance to the simplex.
Therefore, we may focus solely on simplices to control the positioning of general near-extremal bodies.
Lemma~\ref{lem:simplex_iz} already leads to a stability improvement of Theorem~\ref{thm:Lp_RSS} for simplices based on the largest of the $\beta_i$ as shown in Lemma~\ref{lem:simplex_stability} below.
The last remaining step afterward is to connect the $\beta_i$ to the desired statements about the positioning of the simplex.
Before we can proceed with these steps, we require a simple, albeit technical, estimate.

\begin{lemma}\label{lem:1/2_powers}
Let $t \in (0,\infty)$.
Then
\[
	2^{-\frac{1}{t}} + 2^{-t}
	\leq 1.
\]
\end{lemma}
\begin{proof}
Let $f(t) = 2^{-\frac{1}{t}} + 2^{-t}$ and assume for a contradiction that there exists some $t_0 \in (0,\infty)$ with $f(t_0) > 1$.
Since $f(t) = f(\frac{1}{t})$ for all $t \in (0,\infty)$, we may suppose that $t_0 \in (1,\infty)$.
Then $1 = f(1) = \lim_{t \to \infty} f(t)$ shows that $f$ has a local maximum $t_1 \in (1,\infty)$ with $f(t_1) > 1$.
Using
\[
	0
	= f'(t_1)
	= \ln(2) \bigg( \frac{2^{-\frac{1}{t_1}}}{t_1^2} - 2^{-t_1} \bigg),
\]
we obtain
\[
	1
	< f(t_1)
	= 2^{-\frac{1}{t_1}} + \frac{2^{-\frac{1}{t_1}}}{t_1^2}
	= \frac{1+\frac{1}{t_1^2}}{2^{\frac{1}{t_1}}}.
\]
However, we claim that $1 + t^2 \leq 2^t$ for all $t \in (0,1)$,
in which case the particular choice $t = \frac{1}{t_1}$ leads to the desired contradiction.
Now, let $g(t) = 2^t - 1 - t^2$.
Then $g(0) = g(1) = 0$ and
\[
	g''(t)
	= \ln(2)^2 2^t - 2
	< 2 (\ln(2)^2 - 1)
	< 0
\]
for all $t \in (0,1)$.
Therefore, $g$ is concave on $[0,1]$, so that $g(t) \geq 0$ for all $t \in [0,1]$ as claimed.
\end{proof}

\begin{lemma}\label{lem:simplex_stability}
Let $T \in \CK^n$ be a simplex with $0 \in T$ and $\beta_{n+1} \geq \frac{1}{n+1}$, and suppose that $p > 1$ (and $q < \infty$).
Then
\[
	\vol_{nm}(\diff{p}{m}{T})
	\leq \begin{cases}
		\displaystyle \Bigg( C_{n,m,q} - \frac{nm}{2^{1+\frac{nm}{q}} q} (1-\beta_{n+1}) \Bigg) \vol_n(T)^m,
			& \text{if } \beta_{n+1} \geq 2^{-\frac{1}{q}},
		\vspace{10pt}
		\\
		\displaystyle \Bigg( C_{n,m,q} - \frac{m}{(n+1)^{nm+q} q} \Bigg) \vol_n(T)^m,
			& \text{else}.
	\end{cases}
\]
\end{lemma}
\begin{proof}
From Lemmas~\ref{lem:general_iz}~and~\ref{lem:simplex_iz}, we have for $\mu := \min \{ \beta_{n+1}, 2^{-\frac{1}{q}} \}$ that
\begin{equation}\label{eq:simplex_fiber_estimate}
	\int_{\diff{p}{m}{T}} |\interval{p}{m}{T}{z}| \dd z
	\geq \int_{\diff{p}{m}{T}} 1 - \|z\|_{\diff{p}{m}{T}}^q \dd z
		+ \min \bigg\{ \frac{1-\beta_{n+1}}{\beta_{n+1}}, \frac{\beta_{n+1}}{1-\beta_{n+1}} \bigg\} \int_{\cartesian{\simplexscale{T}{\mu}}{m}} \simplexscaling{T}{z}^q \dd z.
\end{equation}
Again using the layer cake representation for the $q$-norm gives, similar to before,
\begin{align*}
	\int_{\cartesian{\simplexscale{T}{\mu}}{m}} \simplexscaling{T}{z}^q \dd z
	& = q \int_0^\mu t^{q-1} \vol_{nm}(\cartesian{\simplexscale{T}{\mu}}{m} \setminus \cartesian{\simplexscale{T}{t}}{m}) \dd t
	= q \vol_{nm}(\cartesian{T}{m}) \int_0^\mu t^{q-1} (\mu^{nm} - t^{nm}) \dd t
	\\
	& = q \vol_n(T)^m \bigg( \frac{\mu^{nm+q}}{q} - \frac{\mu^{nm+q}}{nm+q} \bigg)
    = \frac{nm}{nm+q} \mu^{nm+q} \vol_n(T)^m.
\end{align*}
Together with \eqref{eq:fiber_integral}--\eqref{eq:simplex_fiber_estimate} and the discussion for the equality case in Theorem~\ref{thm:Lp_RSS}, this results in
\[
	\vol_{nm}(\diff{p}{m}{T}) + \frac{nm}{q} \mu^{nm+q} \vol_n(T)^m \min \bigg\{ \frac{1-\beta_{n+1}}{\beta_{n+1}}, \frac{\beta_{n+1}}{1-\beta_{n+1}} \bigg\}
	\leq C_{n,m,q} \vol_n(T)^m.
\]

The last step to complete the proof is a case distinction for how large $\beta_{n+1}$ is.
If $\beta_{n+1} \geq 2^{-\frac{1}{q}}$, then in particular $\beta_{n+1} \geq \frac{1}{2}$.
Hence, $\mu = 2^{-\frac{1}{q}}$, and $\beta_{n+1} \leq 1$ shows
\[
    \min \bigg\{ \frac{1-\beta_{n+1}}{\beta_{n+1}}, \frac{\beta_{n+1}}{1-\beta_{n+1}} \bigg\}
    = \frac{1-\beta_{n+1}}{\beta_{n+1}}
    \geq 1-\beta_{n+1},
\]
from which the claimed inequality in this case follows.
If instead $\beta_{n+1} \leq 2^{-\frac{1}{q}}$,
then $\mu = \beta_{n+1}$.
It is straightforward to check that $h(\beta) := \beta^{nm+q} \min \{ \frac{1-\beta}{\beta}, \frac{\beta}{1-\beta} \}$
is increasing in $\beta \in [\frac{1}{n+1},1-\frac{1}{nm+q}]$ and decreasing in $\beta \in [1-\frac{1}{nm+q},1]$.
In particular, $h$ attains it minimal value on the interval $[\frac{1}{n+1},2^{-\frac{1}{q}}]$ at one of the endpoints.
With Lemma~\ref{lem:1/2_powers}, we see that
\[
	h(2^{-\frac{1}{q}})
	= 2^{-\frac{nm+q-1}{q}} (1-2^{-\frac{1}{q}})
	\geq 2^{-\frac{nm+q-1}{q} - q}
	\geq 2^{-nm-q}
	= h \Big( \frac{1}{2} \Big)
	\geq h \Big( \frac{1}{n+1} \Big),
\]
so $h$ attains its minimum on $[\frac{1}{n+1},2^{-\frac{1}{q}}]$ at $\frac{1}{n+1}$.
This yields the inequality in the second case.
\end{proof}

The following two lemmas give the aforementioned control over the positioning of $T$ in terms of the largest of the $\beta_i$ when it is at least $\frac{1}{2}$.

\begin{lemma}\label{lem:simplex_symmetrizations}
Let $T \in \CK^n$ be a simplex with $0 \in T$ and $\beta_{n+1} \geq \frac{1}{2}$.
Then
\[
	\max_{z \in T \cap (-T)} \|z\|_{T-T}
	= 1 - \beta_{n+1}.
\]
\end{lemma}
\begin{proof}
If $m = p = 1$ (and $q = \infty$), we have $\lyz{p}{m}{T}{\frac{1}{2}} = T-T$.
Therefore, Lemma~\ref{lem:simplex_lyz} together with $\sum_{i \in [n+1]} \beta_i = 1$ shows for $z \in \R^n$ and $\tau \geq 0$ that $\tau z \in T-T$ if and only if
\[
    0
    \leq \sum_{i \in [n+1]} \min \big\{ \beta_i, \beta_i + (a^i)^T (\tau z) \big\}
    = 1 + \tau \sum_{i \in [n+1]} \min \big\{ 0, (a^i)^T z \big\}.
\]
If $z$ is non-zero, then the largest $\tau \geq 0$ satisfying these conditions is $\|z\|_{T-T}^{-1}$ by the former, and $(- \sum_{i \in [n+1]} \min\{ 0, (a^i)^T z \})^{-1}$ by the latter.
Consequently, for all $z \in \R^n$,
\[
    \|z\|_{T-T}
    = \sum_{i \in [n+1]} \max \big\{ 0, -(a^i)^T z \big\}.
\]

If we now want to estimate $\|z\|_{T-T}$ any $z \in (T \cap (-T))$,
we may assume that $(a^{n+1})^T z \geq 0$.
Since $z \in T \cap (-T)$ shows $|(a^i)^T z| \leq \beta_i$ for all $i \in [n+1]$, combining this with $\sum_{i \in [n+1]} \beta_i = 1$ gives
\begin{equation}\label{eq:T-T_norm_estimate}
    \|z\|_{T-T}
    = \sum_{i \in [n]} \max \big\{ 0, -(a^i)^T z \big\}
    \leq \sum_{i \in [n]} \beta_i
    = 1 - \beta_{n+1}.
\end{equation}
To see that equality can be achieved, we choose $z \in \R^n$ with $(a^i)^T z = -\beta_i$ for all $i \in [n]$.
Then
\[
    (a^{n+1})^T z
    = - \sum_{i \in [n]} (a^i)^T z
    = \sum_{i \in [n]} \beta_i
    = 1 - \beta_{n+1}
    \leq \beta_{n+1}
\]
from $\sum_{i \in [n+1]} a^i = 0$ and $\beta_{n+1} \geq \frac{1}{2}$ shows $z \in T \cap (-T)$.
This point indeed satisfies \eqref{eq:T-T_norm_estimate} with equality, so the proof is complete.
\end{proof}

\begin{lemma}\label{lem:simplex_reflection}
Let $T \in \CK^n$ be a simplex with $0 \in T$ and $\beta_{n+1} \geq \frac{1}{2}$.
Then
\[
    0
    \in \bd \big( c_T + ( n - (1-\beta_{n+1}) (n+1)) (c_T - T) \big).
\]
\end{lemma}
\begin{proof}
For $\rho \geq 0$ we have $0 \in c_T + \rho (c_T - T)$ if and only if $h_{c_T + \rho (c_T - T)}(-a^i) \geq 0$ for all $i \in [n+1]$.
Using \eqref{eq:simplex_normalization_centered}, this is equivalent to
\[
    0
    \leq \rho h_T(a^i) - (\rho + 1) (a^i)^T c_T
    = \rho \beta_i - (\rho + 1) \bigg( \beta_i - \frac{1}{n+1} \bigg)
    = \frac{\rho+1}{n+1} - \beta_i.
\]
Since $\sum_{i \in [n+1]} \beta_i = 1$ together with $\beta_i \geq 0$ for $i \in [n]$ by $0 \in T$ and $\beta_{n+1} \geq \frac{1}{2}$ by assumption implies $\beta_{n+1} \geq \beta_i$
for all $i \in [n]$,
the condition $0 \in c_T + \rho (c_T - T)$ is equivalent to just
\[
    \rho
    \geq \beta_{n+1} (n+1) - 1
    = n - (1 - \beta_{n+1}) (n+1).
\]
The claim follows.
\end{proof}

We are finally ready to prove Theorem~\ref{thm:Lp_RSS_stability}.

\begin{proof}[Proof of Theorem~\ref{thm:Lp_RSS_stability}]
The assumption on $K$ and the proof of Theorem~\ref{thm:Lp_RSS} show
\[
    C_{n,m,q} (1-\varepsilon)
    \leq \frac{\vol_{nm}(\diff{p}{m}{K})}{\vol_n(K)^m}
    \leq C_{n,m,q} - nm \binom{nm/q}{1/q}^{-1} (n-s(K))
\]
and therefore
\[
    n - s(K)
    \leq \gamma_{n,m,q} \cdot \varepsilon
    < 1.
\]
Applying Theorem~\ref{thm:main} now yields the first claim.

For the rest of the proof, we may
assume $p > 1$, $q < \infty$, and $\varepsilon \in (0,c_{n,m,q}]$.
By the already proven first claim, there exist a simplex $T \in \CK^n$ and a vector $u \in \R^n$ with
\begin{equation}\label{eq:Lp_simplex_containments}
    u + \frac{1}{d_{n,m,q}(\varepsilon)} T
    \subset K
    \subset T.
\end{equation}
In particular, $0 \in K \subset T$,
and the containment from left to right requires $u \in (1 - \frac{1}{d_{n,m,q}(\varepsilon)}) T$.
Now,
\[
    \diff{p}{m}{K} \subset \diff{p}{m}{T}
        \quad \text{and} \quad
    \vol_n(T) \leq d_{n,m,q}(\varepsilon)^n \vol_n(K)
\]
show
\[
    \vol_{nm}(\diff{p}{m}{T})
    \geq \vol_{nm}(\diff{p}{m}{K})
    \geq C_{n,m,q} (1-\varepsilon) \vol_n(K)^m
    \geq C_{n,m,q} \frac{1-\varepsilon}{d_{n,m,q}(\varepsilon)^{nm}} \vol_n(T)^m.
\]
The simple inequalities $\frac{1}{(1+x)^{nm}} \geq 1 - nmx$ for $x \in (0,\infty)$,
as well as $C_{n,m,q} \geq \gamma_{n,m,q}$ and $C_{n,m,q} \geq 2$ from $\binom{nm/q}{k/q} \leq \binom{nm}{k}$ for all $k \in \{0\} \cup [n]$, together with the stronger assumption on $\varepsilon$,
further give
\begin{align*}
    \frac{1-\varepsilon}{d_{n,m,q}(\varepsilon)^{nm}}
    & \geq (1-\varepsilon) \bigg( 1 - nm C_{n,m,q} \cdot \varepsilon - \frac{nm C_{n,m,q}^2 \cdot \varepsilon^2}{2 (1 - C_{n,m,q} \cdot \varepsilon)} \bigg)
    \\
    & > 1 - \bigg( 1 + nm C_{n,m,q} + \frac{nm C_{n,m,q}^2 \cdot \varepsilon}{2 (1 - C_{n,m,q} \cdot \varepsilon)} \bigg) \varepsilon
    \\
    & \geq 1 - (1 + nm C_{n,m,q} + 1) \varepsilon
    \geq 1 - (n+1) m C_{n,m,q} \cdot \varepsilon.
\end{align*}
If we re-index the vertices of $T$ such that $\beta_{n+1} = \max_{i \in [n+1]} \beta_i$, then $\sum_{i \in [n+1]} \beta_i = 1$ shows $\beta_{n+1} \geq \frac{1}{n+1}$.
Assuming for a contradiction that $\beta_{n+1} < 2^{-\frac{1}{q}}$,
Lemma~\ref{lem:simplex_stability} would then imply
\[
    C_{n,m,q} (1 - (n+1) m C_{n,m,q} \cdot \varepsilon)
    < \frac{\vol_{nm}(\diff{p}{m}{T})}{\vol_n(T)^m}
    \leq C_{n,m,q} - \frac{m}{(n+1)^{nm+q} q}
\]
and therefore
\[
    \frac{m}{(n+1)^{nm+q}q}
    < (n+1) m C_{n,m,q}^2 \cdot \varepsilon.
\]
However, this contradicts the stronger assumption on $\varepsilon$,
so $\beta_{n+1} \geq 2^{-\frac{1}{q}}$.
Lemma~\ref{lem:simplex_stability} now gives
\[
    C_{n,m,q} (1 - (n+1) m C_{n,m,q} \cdot \varepsilon)
    < \frac{\vol_{nm}(\diff{p}{m}{T})}{\vol_n(T)^m}
    \leq C_{n,m,q} - \frac{nm}{2^{1+\frac{nm}{q}} q} (1-\beta_{n+1})
\]
and in particular
\[
    1-\beta_{n+1}
    < \xi_{n,m,q} \cdot \varepsilon.
\]
We conclude from \eqref{eq:Lp_simplex_containments}, Lemma~\ref{lem:simplex_symmetrizations}, and the assumption on $\varepsilon$ that
\[
    \max_{z \in K \cap (-K)} \|z\|_{K-K}
    \leq d_{n,m,q}(\varepsilon) \max_{z \in T \cap (-T)} \|z\|_{T-T}
    < 2 \xi_{n,m,q} \cdot \varepsilon.
\]
Lastly, Lemma~\ref{lem:simplex_reflection} shows for $\rho \in [0, n - (1 - \beta_{n+1})(n+1))$ that $0 \notin \inte(c_T + \rho (c_T - T))$.
With
\begin{align*}
	(1 - \beta_{n+1}) (n+1)
    < (n+1) \xi_{n,m,q} \cdot \varepsilon
    \leq n,
\end{align*}
all claims up to the optimality of the linear orders in $\varepsilon \leq c_{n,m,q}$ are verified.

For the optimality of the first estimate, let $L \in \CK^2$ be a trapezoid that has one vertex of its longer parallel side at $0$.
While the current body $L$ is a subset of $\R^k$ for some $k < n$, we replace it by a cone of 
the form $\conv( (L \times \{0\}) \cup \{x\})$ with $x \in \R^{k+1} \setminus \aff(L \times \{0\})$.
Let $K \in \CK^n$ be the convex body obtained at the end of this process.
Then Propositions~\ref{prop:trapezoid}~and~\ref{prop:cone_distance} together show that there exists a simplex $T$ also with a vertex at $0$ such that $T \subset K \subset d_{BM}(K,\simplex{n}) T$.
Consequently, by Theorem~\ref{thm:Lp_RSS},
\[
    \vol_{nm}(\diff{p}{m}{K})
    \geq \vol_{nm}(\diff{p}{m}{T})
    = C_{n,m,q} \vol_n(T)^m
    \geq \frac{C_{n,m,q}}{d_{BM}(K,\simplex{n})^{nm}} \vol_n(K)^m.
\]
If we choose $\varepsilon \geq 0$ such that $\vol_{nm}(\diff{p}{m}{K}) = C_{n,m,q} (1-\varepsilon) \vol_n(K)^m$, this implies
\[
    d_{BM}(K,\simplex{n})
    \geq \frac{1}{(1-\varepsilon)^{\frac{1}{nm}}}
    \geq 1 + \frac{\varepsilon}{nm}.
\]

For the other two estimates, let $K = T$ be a simplex with $\beta_{n+1} \geq \frac{1}{2}$.
Then $0 \in \simplexscale{T}{\beta_{n+1}} \subset T$ by the proof of Lemma~\ref{lem:simplex_iz}, so Theorem~\ref{thm:Lp_RSS} gives
\[
    \vol_{nm}(\diff{p}{m}{T})
    \geq \vol_{nm}(\diff{p}{m}{\simplexscale{T}{\beta_{n+1}}})
    = C_{n,m,q} \vol_n(\simplexscale{T}{\beta_{n+1}})^m
    = C_{n,m,q} \beta_{n+1}^{nm} \vol_n(T)^m.
\]
In particular, if we choose $\varepsilon := nm (1-\beta_{n+1})$, then
\[
    \vol_{nm}(\diff{p}{m}{T})
    \geq C_{n,m,q} \bigg( 1 - \frac{\varepsilon}{nm} \bigg)^{nm} \vol_n(T)^m
    \geq C_{n,m,q} (1-\varepsilon) \vol_n(T)^m.
\]
Lemma~\ref{lem:simplex_symmetrizations} shows that there exists some $v \in T \cap (-T)$ with $\|v\|_{T-T} = 1 - \beta_{n+1} = \frac{\varepsilon}{nm}$, and Lemma~\ref{lem:simplex_reflection} shows that $0 \in \inte( c_T + \rho (c_T - T))$ for all $\rho > n - (1-\beta_{n+1})(n+1) = n - \frac{n+1}{nm} \varepsilon$.
Hence, the linear order in the corresponding estimates in the theorem is optimal, completing the proof.
\end{proof}

\section{Diameter of the Banach\texorpdfstring{--}{-}Mazur Compactum}\label{sec:DiameterBMcompactum}

We end this paper with our bound on the maximal Banach--Mazur distance between convex bodies in $\R^n$.
Note that an exact computation of the best upper bound obtainable from the following proof is theoretically possible, as it comes down to finding zeros of a fourth degree polynomial.
However, the resulting expression would be rather unwieldy, so that it is more practical to give an estimate for the upper bound instead.

\begin{proof}[Proof of Theorem~\ref{thm:BMdistance}]
Let us assume without loss of generality that $s(L) \geq s(K)$ and write $\varepsilon := n - s(K)$.
A direct application of \cite[Corollary~$5.10$]{2004:GoLiMePa} gives
\begin{equation} \label{eq:max_dist_glmp_eps}
    d_{BM}(K,L)
    \leq n s(K)
    = n (n - \varepsilon),
\end{equation}
so this distance is at most $n^2 - n$ in case $\varepsilon \geq 1$.
From now on, we may suppose that $\varepsilon \in [0,1)$.
Since $s(L) \geq s(K) = n - \varepsilon$,
Theorem~\ref{thm:main} shows that
\begin{equation} \label{eq:max_dist_stab_eps}
    d_{BM}(K,L)
    \leq d_{BM}(K,\simplex{n}) \cdot d_{BM}(L,\simplex{n})
    \leq \bigg(1 + \varepsilon + \frac{\varepsilon^2}{2(1-\varepsilon)} \bigg)^2.
\end{equation}
Note that $(1+\varepsilon+\frac{\varepsilon^2}{2(1-\varepsilon)})^2$ is an increasing function in $\varepsilon$ that takes the value $1$ for $\varepsilon = 0$ and tends to $\infty$ for $1 > \varepsilon \to 1$.
In contrast, $n (n - \varepsilon)$ is decreasing in $\varepsilon$ with values $n^2$ and $n (n-1)$ for $\varepsilon = 0$ and $\varepsilon = 1$, respectively.
Thus, there exists a unique $\varepsilon(n) \in [0,1)$ where the two functions coincide.
Taking \eqref{eq:max_dist_glmp_eps} and \eqref{eq:max_dist_stab_eps} together now gives
\begin{equation}\label{eq:minDist}
    d_{BM}(K,L)
    \leq \max_{\varepsilon \in [0,1)} \; \min \Bigg\{ n(n-\varepsilon), \bigg(1+\varepsilon+\frac{\varepsilon^2}{2(1-\varepsilon)}\bigg)^2 \Bigg\}
    = n (n-\varepsilon(n)).
\end{equation}
Moreover, notice that $1 - \frac{1}{2n} - \frac{11}{10n^2} \leq \varepsilon(n) \leq 1 - \frac{1}{2n}$.
Indeed, this can be checked by a numerical computation for $n = 1$.
When $n \geq 2$, we first note that $n (n-\varepsilon)$ is still larger than $(1+\varepsilon+\frac{\varepsilon^2}{2(1-\varepsilon)})^2$ for $\varepsilon = 1 - \frac{1}{2n} - \frac{11}{10n^2}$:
\begin{align*}
    & n \bigg( n-1+\frac{1}{2n} + \frac{11}{10n^2} \bigg) - \Bigg( 1+1-\frac{1}{2n} - \frac{11}{10n^2} + \frac{(1-\frac{1}{2n}-\frac{11}{10n^2})^2}{2(\frac{1}{2n} + \frac{11}{10n^2})} \Bigg)^2
    \\
    & = \frac{14000 n^7 - 39600 n^6 - 32400 n^5 + 80775 n^4 + 120340 n^3 + 35090 n^2 - 26620 n - 14641}{400 n^4 (5 n + 11)^2},
\end{align*}
where the seventh derivative in $n$ of the numerator is clearly positive at any value $n \geq 2$.
From this, it suffices to evaluate all derivatives of the numerator from order $6$ down to $0$ at the value $n = 2$ to verify that all of them are positive at any value $n \geq 2$.
In particular, the numerator itself, and thus the entire quotient, is positive for all $n \geq 2$.
In contrast, $n (n - \varepsilon)$ is smaller than $(1+\varepsilon+\frac{\varepsilon^2}{2(1-\varepsilon)})^2$ for $\varepsilon = 1 - \frac{1}{2n}$:
\[
    n \bigg( n-1+\frac{1}{2n} \bigg) - \Bigg( 1+1-\frac{1}{2n} + \frac{(1-\frac{1}{2n})^2}{2 \frac{1}{2n}} \Bigg)^2
    = \frac{n (8 - 48 n^2) - 1}{16 n^2},
\]
which is clearly negative for all $n \geq 2$.
In conclusion,
\[
    d_{BM}(K,L)
    \leq b(n)
    := n (n - \varepsilon(n))
    \in \bigg( n^2 - n + \frac{1}{2}, n^2 - n + \frac{1}{2} + \frac{11}{10n} \bigg).
\]

If we equal both terms appearing within the minimum in \eqref{eq:minDist} and solve it with computer aid, we obtain accurate values of $\varepsilon(n)$ and $b(n)$, which in turn are kept for $3 \leq n \leq 8$ within Table~\ref{tab:bn}.
The value used in the table for $n = 2$ is taken from \cite{2018:BrodiukPalkoPrymak}.
\end{proof}

\section*{Acknowledgements}

\noindent This research has been funded by the PID2022-136320NB-I00 project / AEI/10.13039/501100011033/ FEDER, UE.

\noindent Furthermore, we would like to thank Alexander Litvak for pointing out a construction of his own that lead to the stated sharpness in Theorem~\ref{thm:AsymmetryDistance}, more precisely to Remark~\ref{rem:Litvak}.

\printbibliography

\bigskip

René Brandenberg --
Technical University of Munich, Department of Mathematics, Germany. \\
\textbf{rene.brandenberg@tum.de}

Bernardo Gonz\'alez Merino --
Universidad de Murcia, Departmento de Ingenier\'ia y Tecnolog\'ia de Computadores, Spain. \\
\textbf{bgmerino@um.es}

Florian Grundbacher --
Technical University of Munich, Department of Mathematics, Germany. \\
\textbf{florian.grundbacher@tum.de}

\vfill\eject

\end{document}